\documentclass[11pt, reqno]{amsart}
\usepackage{amssymb, hyperref}
\usepackage{amsthm}
\usepackage{amsmath}
\usepackage{graphicx}
\usepackage[all,2cell]{xy}
\usepackage{verbatim}
\usepackage{tikz}
\usepackage{tikz-cd}
\usepackage{hyperref}
\usepackage{array}
\usepackage[capitalise]{cleveref}
\usepackage{epsfig,graphicx}
\usepackage{tikz}
\usepackage{tikz-cd}
\usepackage{float}
\usepackage{capt-of}
\tikzset{node distance=2cm, auto}
\usepackage[vcentermath]{youngtab}
\usepackage{ytableau}
\usepackage{hyperref}
\usepackage{etoolbox}
\usetikzlibrary{graphs}
\usepackage{listings}
\usetikzlibrary{backgrounds}
\usetikzlibrary{arrows,decorations.markings}
\tikzstyle{vertex}=[circle, draw, inner sep=0pt, minimum size=6pt]

\usetikzlibrary{matrix,arrows,decorations.pathmorphing}
\usepackage{mathrsfs}
\usepackage{caption}
\numberwithin{equation}{section}
\newtheorem*{theorem*}{Theorem}
\newtheorem*{corollary*}{\bf Corollary}
\newtheorem*{remark*}{\bf Remark}
\usepackage{lipsum}
\usepackage{lineno}
\newtheorem{theorem}{Theorem}[section]
\newtheorem{corollary}[theorem]{Corollary}

\newtheorem{lemma}[theorem]{Lemma}

\newtheorem{proposition}[theorem]{Proposition}
\newtheorem{remark}[theorem]{Remark}

\newcommand{\eat}[1]{}

\title{On torus quotients of Schubert varieties in orthogonal Grassmannian-II}
\author[A. Nayek]{Arpita Nayek}
\address{Arpita Nayek\\
	Department of Mathematics\\ IIT Bombay, Powai, Mumbai\\ 400076, India.}
\email{arpitan@math.iitb.ac.in}
\author[P. Saha]{Pinakinath Saha}
\address{Pinakinath Saha\\
	Department of Mathematics\\ IIT Bombay, Powai, Mumbai \\ 400076, India.}
\email{psaha@math.iitb.ac.in}

\begin{document}
	\begin{abstract}	
		Let $G=SO(8n+4,\mathbb{C})$ ($n\ge 1$). Let $B$ be a Borel subgroup of $G$ containing a maximal torus $T$ of $G.$ Let $P (\supset B)$ denote the maximal parabolic subgroup of $G$ corresponding to the simple root $\alpha_{4n+2}$. In this article, we prove projective normality of the GIT quotients of certain Schubert varieties in the orthogonal Grassmannian $G/P$ with respect to the descent of a suitable $T$-linearized very ample line bundle. 
	\end{abstract}
	\maketitle
	\section{Introduction}\label{section1}
	Let $G=SO(8n+4,\mathbb{C}).$ Let $B$ be a Borel subgroup of $G$ containing a maximal torus $T$ of $G.$ Let $W$ be the Weyl group of $G$ with respect to $T.$ Let $P(\supset B)$ denote the maximal parabolic subgroup of $G$ corresponding to the simple root $\alpha_{4n+2}$. Let $\varpi_{4n+2}$ be the fundamental weight corresponding to the simple root $\alpha_{4n+2}.$
	Let $W_{P}$ be the Weyl group of $P.$
	Let $W^{P}$ denote the set of minimal coset representatives of $W/W_{P}.$ Let $``\le"$ denote the Bruhat-Chevalley order on $W^{P}.$ For $w\in W^{P},$ let $X(w):=\overline{BwP/P}$ denote the Schubert variety in $G/P$ corresponding to $w.$  
		
	To state the main result of our article we introduce some notations.
	
	For $1\le i\le 2n,$ let $\tau_{2i-1}:=\bigg\{ \begin{matrix}   s_{2i-1}s_{2i}\cdots s_{4n-1}s_{4n}s_{4n+1} & \text{~for~ $i$ ~even}\\ s_{2i-1}s_{2i}\cdots s_{4n-1}s_{4n}s_{4n+2} &\text{~ for $i$~ odd}
	\end{matrix}$
\noindent where $s_{i}$ denotes the simple reflection corresponding to the simple root $\alpha_{i}.$ 

Let $\lambda=\varpi_{4n+2}.$ Let $v_1=s_{4n+2}\prod_{i=2n}^{1}\tau_{2i-1}.$ Then $v_1$ is a unique minimal element in $W^{P}$ such that $v_1(\lambda)\leq 0.$ Consider the following $v_i$'s such that $v_1 \leq v_i$ for all $2 \leq i \leq 6:$
	\begin{center}
		\begin{tikzpicture}[scale=.5]
			\node (e) at (0,0)  {$v_1$};
			\node (z) at (-2.2,-1.2) {$s_{4n-2}$};
			\node (y) at (2,-1.2) {$s_{4n}$};
			\node (g) at (-3,-3) {$v_2$};
			\node (b) at (3,-3) {$v_3$};
			\node (w) at (5.7,-4.8) {$s_{4n+1}$};
			\node (v) at (-2.5,-4.8) {$s_{4n}$};
			\node (u) at (2.5,-4.8) {$s_{4n-2}$};
			\node (d) at (0,-6) {$v_4$};
			\node (p) at (6,-6) {$v_5$};
			\node (s) at (5.5,-7.6) {$s_{4n-2}$};
			\node (q) at (.5,-7.6) {$s_{4n+1}$};
			\node (a) at (3,-9) {$v_6$};
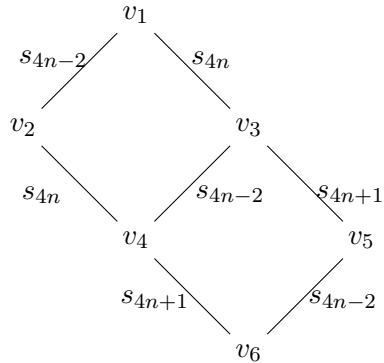
			\node (z) at (3,-10) {Figure 1: Sublattice of the Bruhat lattice $W^{P}$};
			\draw (e) -- (g) -- (d) -- (a) -- (p) -- (b) -- (e);
			\draw (b) -- (d);
		\end{tikzpicture}
	\end{center}
	$v_2=s_{4n-2}v_1, v_3=s_{4n}v_{1}, v_{4}=s_{4n-2}s_{4n}v_{1}, v_{5}=s_{4n+1}s_{4n}v_{1}, v_{6}=s_{4n-2}s_{4n+1}s_{4n}v_{1}.$
	Since $v_1 \leq v_{i}$ and $v_{1}(\lambda)\leq 0,$ we have $v_{i}(\lambda) \leq 0$ for all $2 \leq i \leq 6.$ Thus, by \cite[Lemma 2.1, p.470]{KP} $X(v_i)^{ss}_T(\mathcal{L}(2\lambda))\neq \emptyset$ for all $1 \leq i \leq 6.$ 
	
	In this article, our main results are the following.
	\begin{theorem}\label{cor3.8}
		The GIT quotient $T \backslash \backslash (X(v_6))^{ss}_{T}(\mathcal{L}(4\lambda))$ is projectively normal with respect to the descent of the $T$-linearized very ample line bundle $\mathcal{L}(4\lambda).$
	\end{theorem}
	
	\begin{proposition}\label{prop1.1} 
		\begin{itemize}
			\item[(i)] The GIT quotient $T\backslash\backslash(X(v_1))^{ss}_{T}(\mathcal{L}(4\lambda))$ is a point.
			
			\item[(ii)] The GIT quotient $T\backslash\backslash(X(v_2))^{ss}_{T}(\mathcal{L}(4\lambda))$ is isomorphic to $(\mathbb{P}^{1},\mathcal{O}_{\mathbb{P}^{1}}(2))$ as a polarized variety.
			
			\item[(iii)] The GIT quotient $T\backslash\backslash(X(v_3))^{ss}_{T}(\mathcal{L}(4\lambda))$ is isomorphic to $(\mathbb{P}^{1},\mathcal{O}_{\mathbb{P}^{1}}(2))$ as a polarized variety.
			
			\item[(iv)] The GIT quotient $T\backslash\backslash(X(v_4))^{ss}_{T}(\mathcal{L}(4\lambda))$ is isomorphic to $(\mathbb{P}^{3},\mathcal{O}_{\mathbb{P}^{3}}(2))$ as a polarized variety.

			\item[(v)] The GIT quotient $T\backslash\backslash(X(v_5))^{ss}_{T}(\mathcal{L}(4\lambda))$ is isomorphic to $(\mathbb{P}^{2},\mathcal{O}_{\mathbb{P}^{2}}(2))$ as a polarized variety.
		\end{itemize}	
	\end{proposition}
	
	When $G=SO(8n, \mathbb{C})(n\ge 1),$ analogous statements of \cref{cor3.8} and \cref{prop1.1} are proved in \cite{NS-I}.
	
	The layout of the article is as follows. In \cref{section2}, we introduce some notation and recall some preliminaries. In \cref{section3}, we prove \cref{cor3.8} and \cref{prop1.1}.
	
	\section{Notation and Preliminaries}\label{section2}
	We refer \cite{Hum1}, \cite{Hum2}, \cite{Jan}, \cite{LS} and \cite{Mumford} for preliminaries in algebraic groups, Lie algebras, Standard Monomial Theory and Geometric Invariant Theory.
	
	\subsection{Prjective normality}
	A polarized variety $(X,\mathcal{L}),$ where $\mathcal{L}$ is a very-ample line bundle on a projective variety $X$ is said to be {\em projectively normal} if its homogeneous coordinate ring $\bigoplus_{k\in \mathbb{Z}_{\geq 0}} H^{0}(X,\mathcal{L}^{\otimes k})$ is integrally closed and is generated as a $\mathbb{C}$-algebra by $H^0(X,\mathcal{L})$ (see \cite[Chapter II, Exercise 5.14, p.126]{R}).
	
	\subsection{Geometric Invariant Theory}
	Let $G,P$ and $T$ be as in the previous section.
	
	Let $\mathcal{L}$ be a $T$-linealized very ample line bundle on $G/P.$ We also denote the restriction of the line bundle $\mathcal{L}$ on $X(w)$ by $\mathcal{L}.$ 
	
	A point $p \in X(w)$ is said to be a {\em semi-stable} point with respect to the $T$-linearized line bundle $\mathcal{L}$ if there is a $T$-invariant section $s \in H^0(X(w),\mathcal{L}^{\otimes m})$ for some positive integer $m$ such that $s(p)\neq 0.$ We denote the set of all semi-stable points of $X(w)$ with respect to $\mathcal{L}$ by $X(w)^{ss}_{T}(\mathcal{L}).$ 
	
	A point $p$ in $X(w)^{ss}_{T}(\mathcal{L})$ is said to be a {\em stable} point if $T$-orbit of $p$ is closed in $X(w)^{ss}_{T}(\mathcal{L})$ and stabilizer of $p$ in $T$ is finite. We denote the set of all stable points of $X(w)$ with respect to $\mathcal{L}$ by $X(w)^{s}_{T}(\mathcal{L}).$

	\subsection{Notation and preliminaries for $SO(8n+4, \mathbb{C})$} 
	
	Let us recall some well-known properties of $SO(8n+4, \mathbb{C})$ (see \cite{LS}).
	Let $V= \mathbb{C}^{8n+4}$ together with a non-degenerate symmetric bilinear form $(-,-).$ Let $\{v_{1},\ldots , v_{8n+4} \}$ be the standard basis of $V.$ Let $E=\begin{pmatrix}
		0& J\\
		J&0\\
	\end{pmatrix}$ be the matrix of
	the form $(-,- )$ with respect to the standard basis $\{v_{1},\ldots , v_{8n+4} \},$ where $J$ denotes the anti-diagonal matrix of size $4n+2 \times 4n+2$ with all the entries are $1.$ 
	
	Let $H=SL(V).$ Let $\sigma : H\longrightarrow H$ be the involution defined by $\sigma(A)= E(A^{t})^{-1}E,$
	where $A\in H$ and $A^{t}$ denote transpose of $A.$ Let $G={H}^{\sigma}.$ Then $G=SO(V).$ 
	Let $T_{H}$ (resp. $B_{H}$) be the maximal torus in $H$ consisting of diagonal matrices (resp. the Borel subgroup in $H$ consisting of the upper triangular matrices). Let $T=T_{H}^{\sigma}$ and $B=B_{H}^{\sigma}.$ Then $B$ is a Borel subgroup of $G$ containing $T$ as a maximal torus of $G.$ Let $N_{G}(T)$ (resp. $N_{H}(T_{H})$) be the normalizer in $G$ (resp. $H$) of $T$ (resp. $T_{H}$). Then we have $N_{G}(T)\subseteq N_{H}(T_{H}).$ 
	Thus, we obtain a homomorphism $$W\longrightarrow W_{H},$$ where $W,$ $W_{H}$ denote the Weyl groups of $G,$ $H$ with respect to $T,$  $T_{H}$ respectively. 
	
	\subsubsection{Weyl group}
	The Weyl group $W_{H}$ is identified with $S_{8n+4},$ the symmetric group on the $8n+4$ letters $\{1,2,\ldots,8n+4\}.$
	Further, $\sigma$ induces an involution $W_{H}\longrightarrow W_{H}$ given by
	$$w=(a_{1},\ldots , a_{8n+4})\mapsto \sigma(w)=(c_{1},\ldots , c_{8n+4}),$$
	where $w=(a_{1},\ldots , a_{8n+4})$ is written in one line notation of the permutation and $c_{i} = 8n+5-a_{8n+5-i}.$ Thus, we have $W =\{ w \in W_{H}^{\sigma}: ~w~ \text {is an even permutation in}~ W_{H}  \}.$ In one line notation, we have
	$W=\{(a_{1},\ldots,  a_{8n+4}) \in  S_{8n+4} : a_{i}=8n+5-a_{8n+5-i}, 1 \le i \le 8n+4 \text{ and } ~m_{w}~ \text {is even}\},$
	where $m_{w}=\#\{ i\in \{1,\ldots, 4n+2\}: a_{i} > 4n+2\}.$ Thus, $w=(a_{1},\ldots, a_{8n+4})\in  W$ is known once $(a_{1}, \ldots, a_{4n+2})$ is known.

	\subsubsection{Root System and some other properties} Let $X(T_{H})$ (resp. $X(T)$) be the group of characters of $T_{H}$ (resp. $T$). For $1\le i\le 8n+4,$ the elements $\epsilon_{i}\in X(T_{H})$ are defined by 
	$$\epsilon_{i}(t)=t_i,$$ where $t=diag(t_1, t_2, \ldots, t_{8n+4}) \in T_{H}.$ 
	
	Then the set $R_{H}=\{\epsilon_{i}-\epsilon_j: i\neq j\}$ can be identified with the set of roots of $H$ with respect to $T_{H}$ and that $R_{H}^{+}=\{\epsilon_{i}-\epsilon_{j}: 1\le i<j\le 8n+4\}$ is the set of positive roots with respect to $B_{H}.$ For $1\le i\le 8n+3,$ we denote $\beta_i=\epsilon_{i}-\epsilon_{i+1}.$ Then $S_{H}=\{\beta_{i}: 1\le i\le 8n+3\}$ is the set of simple roots of $R_{H}.$ 
	
	Note that $\sigma$ induces an involution 
	$\sigma: X(T_{H})\longrightarrow X(T_{H}),$ on $X(T_{H})$ defined by $\chi\mapsto\sigma(\chi),$ where $\sigma(\chi)(t)=\chi(\sigma(t))$ for $t \in T_{H}.$ 
	Then we have $\sigma(\epsilon_{i})=-\epsilon_{8n+5-i}$ for $1\le i\le 8n+4.$ Since $T\subset T_{H},$ there is a surjective map $\varphi: X(T_{H})\longrightarrow X(T),$ defined by $\varphi(\epsilon_{i})=-\varphi(\epsilon_{8n+5-i})$ for $1\le i\le 8n+4.$

	Then $\sigma$ leaves $R_{H}$ (resp. $R_{H}^{+}$) stable. Let $R$ (resp. $R^{+}$) be the set of roots (resp. positive roots) of $G$ with respect to $T$ (resp. $B$). Then from the explicit nature of the adjoint representation of $G$ on $\mathfrak{g}$ (Lie algebra of $G$) it follows that $R=\varphi(R_{H}\setminus R_{H}^{\sigma})$ and $R^{+}=\varphi(R_{H}^{+}\setminus {R_{H}^{+}}^{\sigma}).$ In particular, we have  $R$ (resp. $R^{+}$) can be  identified with the orbit space of $R_{H}$ (resp. $R_{H}^{+}$) under the action of $\sigma$ minus the fixed point set under $\sigma.$ We see now that $R^{+}$ can be identified with the subset $\{\epsilon_{i}-\epsilon_j: 1\le i<j\le 4n+2\}\cup \{\epsilon_{i} + \epsilon_{j}:1\le i<j\le 4n+2\}$ of $X(H)$ through $\varphi.$ For $1\le i\le 4n+1,$ we denote $\alpha_{i}=\epsilon_{i}-\epsilon_{i+1}$ and $\alpha_{4n+2}=\epsilon_{4n+1}+\epsilon_{4n+2}.$ Then $S=\{\alpha_{i}: 1\le i\le 4n+2\}$ is the set of simple roots of $R.$
	
    Let $Q$ (resp. $P$) denote the maximal parabolic subgroup of $H$ (resp. $G$) corresponding to the simple root $\beta_{4n+2}$ (resp. $\alpha_{4n+2}$). Let $W_{Q}$ (resp. $W_{P}$) be the Weyl group of the $Q$ (resp. $P$). Then the Schubert varieties in $H/Q$ (resp. $G/P$) are parameterized by the minimal coset representatives of $W_{H}/W_{Q}$ (resp. $W/W_{P}$). We denote the set of all minimal coset representatives of $W_{H}/W_{Q}$ (resp. $W/W_{P}$) by $W^{Q}$ (resp. $W^{P}$).
	
	Let $\varpi'_{4n+2}$ (resp. $\varpi_{4n+2}$) be the fundamental weight corresponding to the simple root $\beta_{4n+2}$ (resp. $\alpha_{4n+2}$) in $H$ (resp. in $G$). Note that the very ample generator of the Picard group of $H/Q$ (resp. of $G/P$) is the line bundle $\mathcal{L}(\varpi'_{4n+2})$ (resp. $\mathcal{L}(\varpi_{4n+2})$). Further, the restriction of the very ample line bundle $\mathcal{L}(\varpi'_{4n+2})$ via the inclusion $G/P\subset H/Q$ is the very ample line bundle $\mathcal{L}(2\varpi_{4n+2}).$
	
	Note that the Bruhat-Chevalley order $~``\le"~$ on $W^{P}$ is induced by the Bruhat-Chevalley order $~``\le"~$ on $W^{Q}.$ Define $I_{4n+2,8n+4}:=\{\underline{i}=(i_{1},\ldots, i_{4n+2}):1 \le i_1 < \cdots < i_{4n+2} \le 8n+4 \}.$ Note that there is a natural partial order $``\le "$ on $I_{4n+2,8n+4}$ defined by $\underline{i}\le \underline{j}$ if and only if  $i_{t}\le j_{t}$ for all $1\le t\le 4n+2.$ Then there is a natural order preserving isomorphism between $(W^{Q}, \le)$ and $(I_{4n+2,8n+4},\le).$ Let $\underline{i}:=(i_{1},...,i_{4n+2})\in I_{4n+2,8n+4},$ and $p_{\underline{i}}$ be the associated Pl\"ucker co-ordinate on $G_{4n+2,8n+4}(\simeq H/Q).$ Let $f_{\underline{i}}$ denote the restriction of $p_{\underline{i}}$ to the set $M_{4n+2}$ of matrices of size $4n+2\times 4n+2$ identified with the opposite big cell $O^{-}_{H}$ of $H/Q.$ Then $f_{\underline{i}}(Y)=p({\underline{i}}(A), {\underline{i}}(B)),$ where ${\underline{i}}(A), {\underline{i}}(B)$ are given as follows: Let $r$ be such that $i_{r}\le 4n+2, i_{r+1}>4n+2.$ Let $s=4n+2-r.$ Then $\underline{i}(A)$ is the $s$-tuple given by $(8n+5-i_{4n+2},...,8n+5-i_{r+1}),$ while $\underline{i}(B)$ is the $s$-tuple given by the complement of $(i_{1},...,i_{r})$ in $(1,...,4n+2).$ Then the pair $(\underline{i}(A), \underline{i}(B))$ is the canonical dual pair associated to $\underline{i}.$
	
	Note that $\mathcal{L}(\varpi'_{4n+2})$ is the tautological line bundle on $H/Q$($\simeq G_{4n+2, 8n+4},$ the Grassmannian of $4n+2$-dimensional subspaces of $V$). Then $H^0(H/Q, \mathcal{L}(\varpi'_{4n+2}))$ is the dual of the irreducible $H$-module with highest weight $\varpi'_{4n+2}.$ Let $\mathcal{L}$ be the restriction of $\mathcal{L}(\varpi'_{4n+2})$ to $G/P.$ For $\underline{i}\in I_{4n+2,8n+4},$ let $p'_{\underline{i}}$ be the restriction of $p_{\underline{i}}$ to $G/P.$ Note that $p'_{\underline{i}}\in H^0(G/P, \mathcal{L}).$ Let $p(\underline{i}(A), \underline{i}(B))$ be the restriction of $p_{\underline{i}}$ to $M_{4n+2}$  with the above identification. Let $p'(\underline{i}(A), \underline{i}(B))$ be the restriction of $p(\underline{i}(A), \underline{i}(B))$ to the set $SkM_{4n+2}$ of skew-symmetric matrices of size $4n+2\times 4n+2$ (identified with the opposite big cell $O^{-}$ of $G/P$). Let $Y\in SkM_{4n+2}.$ Consider $\underline{i}\in I_{4n+2,8n+4}$ such that $\underline{i}(A)=\underline{i}(B).$ Then $p'(\underline{i}(A), \underline{i}(A))(Y)$ is a principal minor of $Y.$ Hence, $p'(\underline{i}(A), \underline{i}(A))(Y)$ is a square denoting $q_{\underline{i}}(Y)$ the corresponding Pfaffian (i.e., $q_{\underline{i}}^{2}(Y)=p'(\underline{i}(A), \underline{i}(A))(Y)$). Thus we obtain a regular function 
	$q_{\underline{i}}: SkM_{4n+2}\longrightarrow \mathbb{C}.$ Let $r=\#\underline{i}(A).$ Then we have that $q_{\underline{i}}$ is non zero if only if $r$ is even (since the determinant of a skew-symmetric $r\times r$ matrix is zero, if $r$ is odd). Further, we have $\mathcal{L}^{2}=\mathcal{L}(\varpi'_{4n+2}).$ 
	
	\subsection{Descends to a line bundle}
	Now, we recall the following theorem due to Shrawan Kumar. 
	\begin{theorem}[\cite{Kum}, Theorem 3.10, p.764]\label{thm2.1}
		Let $G=Spin(8n+4, \mathbb{C}).$ Then the line bundle $\mathcal{L}({\lambda})$ descends to a line bundle on the GIT quotient $T\backslash\backslash(G/P)_{T}^{ss}(\mathcal{L}({\lambda}))$ if and only if $\lambda=4m\varpi_{4n+2}$ for some positive integer $m.$ 
	\end{theorem}

	\subsection{Young tableau for $Spin(8n+4,\mathbb{C})$}Now, we present here the simplified version of the definition of Young tableau and standard Young tableau for $Spin(8n+4,\mathbb{C})$ associated to the weight $\lambda=2m\varpi_{4n+2},$ where $m$ is a positive integer (for more general see \cite[Appendix p.363-365]{L}).
	 	
	\subsubsection{Young tableau for $\lambda=2m\varpi_{4n+2},$ where $m \in \mathbb{N}$}\label{subsection2.1} Define $p_i = 2m$ for $1 \leq i \leq 4n+2.$ Associated to $\lambda,$ we define a partition $p(\lambda):=(p_{1},p_{2},\ldots,p_{4n+2}).$ Then a Young diagram of shape $p(\lambda) = (p_1, \ldots, p_{4n+2})$ associated to $\lambda$ is denoted by $\Gamma$ consists of $p_1$ boxes in the first column, $p_2$ boxes  in the second column and so on.
	
	A Young tableau of shape $\lambda$ is a Young diagram filled with integers between $1$ and $8n+4.$ 
	
	A Young tableau is said to be standard if the entries along any column is non-decreasing from top to bottom and along any row is strictly increasing from left to right.
	
	A Young tableau $\Gamma$ of shape $p(\lambda)$ is said to be $Spin(8n+4,\mathbb{C})$-standard if $\Gamma$ is standard and if $r= (i_{1},i_{2},\ldots,i_{4n+2})$ is a row of length $4n+2$ such that if $i_{j}$ is an entry of the row, then $8n+5-i_{j}$ is not an entry of this row. The number of integers greater than $4n+2$ in a row of $\Gamma$ is even. 
	
	Given a $Spin(8n+4,\mathbb{C})$-standard Young tableau $\Gamma$ of shape $\lambda,$ we associate a section $p_{\Gamma}\in H^0(G/P, \mathcal{L}^{\otimes m}(2\varpi_{4n+2}))^{T}$
	which is called standard monomial of degree $m.$ It is well known that standard monomials of shape $\lambda$ forms a basis of $H^0(G/P, \mathcal{L}^{\otimes m}(2\varpi_{4n+2}))^{T}$

	\begin{remark}[\cite{NayekPattanayak} Lemma 3.2, p.5] \label{zeroweight}
		Fix an integer $1\le t\le 8n+4.$ Let $c_{\Gamma}(t)$ denote the number of times $t$ appears in $\Gamma.$ 
		A monomial $p_{\Gamma}\in H^0(G/P, \mathcal{L}(\lambda))$ is $T$-invariant if and only if $c_{\Gamma}(t)=c_{\Gamma}(8n+5-t)$ for all $1 \leq t \leq 8n+4.$
	\end{remark}
	\subsection{Properties of Pfaffian.}
	To proceed further we recall some result on Pfaffian of a skew-symmetric matrix from \cite{DW} which we will use to determine the straightening law.
	
	The Pfaffian Pf$(A)$ of a skew-symmetric $4n+2\times 4n+2$-matrix	$A=(a_{ij})$ (i.e., $a_{ii}=0$ and $a_{ij}=-a_{ji}$ for $i \neq j$) with coefficients in $\mathbb{C}$ is defined to be Pf$(A)^{2}$=det$(A).$ One can describe the Pfaffian equivalently as a polynomial 
	\begin{center}
		Pf$(X_{12},X_{13},\ldots, X_{14n+2}, X_{23},\ldots, X_{24n+2},\ldots, X_{4n+14n+2})\in \mathbb{Z}[X_{12}, X_{13},\ldots, X_{4n+14n+2}]$
	\end{center} such that for each $A=(a_{ij})\in SkM_{4n+2}$ one has 
	
	\begin{center}
		Pf$(a_{12},a_{13},\ldots, a_{14n+2}, a_{23},\ldots, a_{24n+2},\ldots, a_{4n+14n+2})^{2}$=det($A$),
	\end{center}
	normalized such that 
	\begin{equation}
		\text{Pf}(a_{12},a_{13},\ldots, a_{14n+2}, a_{23},\ldots, a_{24n+2},\ldots, a_{4n+14n+2}):=1
	\end{equation}
	for $a_{ij}:=\bigg\{ \begin{matrix}  1 & \text{~for~} j-i=2n+1\\ 0 &\text{~otherwise.}
	\end{matrix}$
	
	To be more explicit, consider as above a skew-symmetric $4n+2\times 4n+2$-matrix $A$ and for each subset $I\subseteq \{1,2, \ldots, 4n+2\}$ of cardinality, say, $m.$ Let $A(I)$ denote the skew-symmetric  $m\times m$-matrix  one gets from $A$ by eliminating all rows and columns not indexed by indices from $I.$ Moreover, put 
	\begin{equation}
		P(I):=\text{Pf}(A(I))
	\end{equation} 
	for every $I\subseteq \{1,2,\ldots,4n+2\},$ where as usual Pf($A(\phi)$):=1. Then the following result holds.
	
	\begin{theorem}[\cite{DW}, Theorem 1, p.122]\label{thm2.3}
		For any two subsets $I_{1}, I_{2}\subseteq \{1,2,\ldots, 4n+2\}$ of odd cardinality and elements $i_{1},i_{2},\ldots,i_{t}\in \{1,2,\ldots,4n+2 \}$ with $i_{1}<i_{2}<\cdots < i_{t}$ and $\{i_{1},\ldots, i_{t}\}=I_{1}\Delta I_{2}:=(I_{1}\setminus I_{2})\cup (I_{2}\setminus I_{1})$ one has \begin{equation}
			\sum_{\tau=1}^{t} (-1)^\tau\cdot P({I_{1}\Delta\{i_{\tau}\}})\cdot P({I_{2}\Delta\{i_{\tau}\}})=0.
		\end{equation}
	\end{theorem}
	
	\begin{remark}\label{rmk2.4}
		For $\underline{i}=(i_1, i_2, \ldots, i_{4n+2})\in I_{4n+2,8n+4},$ we have\\ $q_{\underline{i}}=P(\underline{i}(B))=\ytableausetup
		{boxsize=2.3em}\begin{ytableau}
			i_1 & i_2 & \cdots & i_{4n+2}
		\end{ytableau}.$    
	\end{remark}
	
	\begin{remark}\label{rmk2.5}
		Let $I=\{i_1, i_2, \ldots, i_r\} \subset \{1, 2, \ldots, 4n+2\}$ be such that $i_1<i_2<\cdots <i_r.$ Let $\{j_{1},\ldots, j_{4n+2-r}\}$ be the complements of $\{i_{1},\ldots, i_{r}\}$ in $\{1,2,\ldots, 4n+2\}$ such that $j_1<j_2<\ldots <j_{4n+2-r}.$ Then we have $P(I)=q_{\underline{i}},$ where $\underline{i}=(j_1,j_2,\ldots, j_{4n+2-r}, 8n+5-i_r, \ldots, 8n+5-i_1).$ 
	\end{remark}

	\section{Main Theorem}\label{section3}	
	Recall that by \cite[Corollary 1.9, p.85]{KS} there exists a unique minimal $v_1 \in W^{P}$ such that $v_1(\lambda)\leq 0$. An easy computation shows that $$v_1=s_{4n+2}\prod_{i=2n}^{1}\tau_{2i-1},$$ where $\tau_{2i-1}=s_{2i-1}s_{2i}\cdots s_{4n-1}s_{4n}s_{4n+1}$ for $i$ even and $\tau_{2i-1}=s_{2i-1}s_{2i}\cdots s_{4n-1}s_{4n}s_{4n+2}$ for $i$ odd. In one line notation $v_i$'s $(1 \leq i \leq 6)$ are the following:  
	\begin{itemize}
		\item  $v_1~~=~(2,4,6,\ldots,4n-4,4n-2,4n,4n+3,4n+4,4n+6,4n+8,\ldots,~8n+4)$
		\item $v_2~~=~(2,4,6,\ldots,4n-4,4n-1,4n,4n+3,4n+4,4n+7,4n+8,\ldots,~8n+4)$
		\item $v_3~~=~(2,4,6,\ldots,4n-4,4n-2,4n+1,4n+3,4n+5,4n+6,4n+8,\ldots,~8n+4)$
		\item $v_4~~=~(2,4,6,\ldots,4n-4,4n-1,4n+1,4n+3,4n+5,4n+7,4n+8,\ldots,~8n+4)$
		\item $v_{5}=(2,4,6,\ldots,4n-4,4n-2,4n+2,4n+4,4n+5,4n+6,4n+8,\ldots,~8n+4)$
		\item $v_{6}=(2,4,6,\ldots,4n-4,4n-1,4n+2,4n+4,4n+5,4n+7,4n+8,\ldots,~8n+4).$
	\end{itemize}

	Since $v_1 \leq v_{i}$ and $v_{1}(\lambda)\leq 0,$ we have $v_{i}(\lambda) \leq 0$ for all $2 \leq i \leq 6.$ Thus, by \cite[Lemma 2.1, p.470]{KP} $X(v_i)^{ss}_T(\mathcal{L}(2\lambda))$ is non-empty for all $1 \leq i \leq 6.$

	Let $X= T \backslash \backslash (X(v_{6}))^{ss}_T(\mathcal{L}(2\lambda)).$ Let $R=\oplus_{k \in \mathbb{Z}_{\geq 0}}R_k,$ where $R_{k}=H^0(X(v_{6}), \mathcal{L}^{\otimes k}(2\lambda))^{T}.$  Note that $R_{k}$'s  are finite dimensional vector spaces, and $X= {\rm Proj}(R).$
	
	Let $p_{\Gamma}$ be the standard monomial in $R_k$ associated to a $Spin(8n+4, \mathbb{C})$-standard Young tableau $\Gamma.$ Thus, $\Gamma$ is of the shape $(2k, 2k, \ldots, 2k)$ ($4n+2$ times). Further, the entries along any row of $\Gamma$ are strictly increasing from left to right and entries along any column of $\Gamma$ are non-decreasing from top to bottom (see \cref{subsection2.1}). Since $p_{\Gamma}$ is $T$-invariant, by \cref{zeroweight} we have $ c_\Gamma(t)=c_\Gamma(8n+5-t) \text{ for all } 1 \leq t \leq 8n+4.$ Thus, we have \begin{equation}\label{eq12.2}
		c_{\Gamma}(t)=k 
	\end{equation} 
	for all $1 \leq t \leq 8n+4.$

	Let $\underline{i} \in I_{4n+2,8n+4}$ be such that both $i_{t}$ and $8n+4-i_{t}$ simultaneously does not appear in $\underline{i}$ and the number of integers greater than or equal to $4n+3$ are even. Note that such $\underline{i}$ is uniquely determined by the tuple $(i_1, i_2, \ldots, i_r)$ of increasing sequence of integers less than or equal to $4n+2$ appear in $\underline{i},$ where $r$ is even. We denote the Young tableau of shape $k\lambda$ ($k\ge 1$) with first row $\underline{i}_{1},$ second row $\underline{i}_{2},$ and so on the $k$-th row $\underline{i}_{k}$ by $\text{YT}^{^{k\lambda}}_{_{(i_1(1),i_1(2), \ldots, i_1(r_{1}))~(i_2(1),i_2(2), \ldots, i_2(r_{2}))~\cdots ~(i_k(1),i_k(2), \ldots, i_k(r_{k}))}},$ where the tuple $(i_j(1),i_j(2), \ldots, i_j(r_{j}))$ denotes the increasing sequence of integers less than or equal to $4n+2$ appear in $\underline{i}_{j}.$ 
	
	Consider the following tableaux
	\begin{center}
		$X_1=\text{YT}^{^{2\lambda}}_{_{(1,3,\ldots,4n-3,4n-1,4n+1,4n+2)~(2,4,\ldots,4n-4,4n-2,4n)}}$\\
		$ X_2=\text{YT}^{^{2\lambda}}_{_{(1,3,\ldots,4n-3,4n-2,4n+1,4n+2)~(2,4,\ldots,4n-4,4n-1,4n)}}$\\
		$X_3=\text{YT}^{^{2\lambda}}_{_{(1,3,\ldots,4n-3,4n-1,4n,4n+2)~(2,4,\ldots,4n-4,4n-2,4n+1)}}$\\
		$X_4=\text{YT}^{^{2\lambda}}_{_{(1,3,\ldots,4n-3,4n-2,4n,4n+2)~(2,4,\ldots,4n-4,4n-1,4n+1)}}$\\
		$X_5=\text{YT}^{^{2\lambda}}_{_{(1,3,\ldots,4n-3,4n-1,4n,4n+1)~(2,4,\ldots,4n-4,4n-2,4n+2)}}$\\
		$X_6=\text{YT}^{^{2\lambda}}_{_{(1,3,\ldots,4n-3,4n-2,4n,4n+1)~(2,4,\ldots,4n-4,4n-1,4n+2)}}$\\
		$Y_1=\text{YT}^{^{4\lambda}}_{_{(1,3,\ldots,4n-3,4n-2,4n-1,4n+2)~(1,3,\ldots,4n-3,4n,4n+1,4n+2)~(2,4,\ldots,4n-4,4n-2,4n)~(2,4,\ldots,4n-4,4n-1,4n+1)}}$\\
		$Y_2=\text{YT}^{^{4\lambda}}_{_{(1,3,\ldots,4n-3,4n-2,4n-1,4n+1)~(1,3,\ldots,4n-3,4n,4n+1,4n+2)~(2,4,\ldots,4n-4,4n-2,4n)~(2,4,\ldots,4n-4,4n-1,4n+2)}}$\\
		$Y_3=\text{YT}^{^{4\lambda}}_{_{(1,3,\ldots,4n-3,4n-2,4n-1,4n,4n+1,4n+2)~(1,3,\ldots,4n-3,4n)~(2,4,\ldots,4n-4,4n-2,4n+1)~(2,4,\ldots,4n-4,4n-1,4n+2)}}$\\
		$Y_4=\text{YT}^{^{4\lambda}}_{_{(1,3,\ldots,4n-3,4n-2,4n-1,4n)~(1,3,\ldots,4n-3,4n,4n+1,4n+2)~(2,4,\ldots,4n-4,4n-2,4n+1)~(2,4,\ldots,4n-4,4n-1,4n+2)}}$\\
		$Z_1=\text{YT}^{^{4\lambda}}_{_{(1,3,\ldots,4n-3,4n-2,4n-1,4n)~(1,3,\ldots,4n-3,4n-2,4n+1,4n+2)~(1,3,\ldots,4n-3,4n,4n+1,4n+2)~(2,4,\ldots,4n-4,4n-2,4n)}}$\\
		$\text{YT}^{^{2\lambda}}_{_{(2,4,\ldots,4n-4,4n-1,4n+1)~(2,4,\ldots,4n-4,4n-1,4n+2)}}$\\
		$ Z_2=\text{YT}^{^{4\lambda}}_{_{(1,3,\ldots,4n-3,4n-2,4n-1,4n,4n+1,4n+2)~(1,\ldots,4n-3,4n-2,4n+1,4n+2)~(1,3,\ldots,4n-3,4n)~(2,4,\ldots,4n-4,4n-2,4n)}}$\\
		$\text{YT}^{^{2\lambda}}_{_{(2,4,\ldots,4n-4,4n-1,4n+1)~(2,4,\ldots,4n-4,4n-1,4n+2)}}.$
	\end{center}
	
	\normalsize{\begin{lemma}\label{lem5.1}
			The homogeneous coordinate ring $R$ of $X$ is generated by $X_i,$  $Y_j,$ $Z_1$ and  $Z_2.$
	\end{lemma}}
	
	\begin{proof}
		
		Let $p_{\Gamma} \in R_k$ be a standard monomial. We show that $p_{\Gamma} = p_{\Gamma_1}p_{\Gamma_2},$ where $p_{\Gamma_1}$ is in $R_1$ or $R_2$ or $R_3.$

		For $1 \leq i \leq 2k$ and $1 \leq j \leq 4n+2,$ let $row_i$ denote the $i$-th row, $e_{i,j}$ denote the $(i,j)$-th entry of $\Gamma.$ Let $n_{t,j}$ denote the number of times $t$ appear in the $j$-th column of $\Gamma.$
		
		Since we are considering $X(v_6),$ we have $row_{2k} \leq v_6.$ Thus, for $1 \leq j \leq 2n-2$ or $2n+4 \leq j \leq 4n+2,$ we have $e_{2k,j} \leq 2j.$ Therefore, for $1 \leq j \leq 2n-2$ or $2n+5 \leq j \leq 4n+2,$ we have \begin{equation*}
			e_{i,j}=\left\{\begin{array}{lr}
				2j-1 & \text{for $1 \leq i \leq k$}\\
				2j & \text{for $k+1 \leq i \leq 2k$.}\\
			\end{array}\right.
		\end{equation*}
		Further, since $e_{2k,2n-2} = 4n-4$ (resp. $e_{1, 2n+5}=4n+9$) we have $e_{i,2n-1}=4n-3$ (resp. $e_{i,2n+4}=4n+8$) for all $1 \leq i \leq k$ (resp. $k+1 \leq i \leq 2k$).

		Note that $e_{2k,2n-1} \leq 4n-1.$ Thus, $e_{2k,2n-1}$ is either $4n-2$ or $4n-1$. Also, since $e_{2k,2n} \leq 4n+2,$ we have $e_{2k,2n}$ is either $4n$, $4n+1$ or $4n+2.$ Thus, the followings are the possibilities for $row_{2k}$ 
		\begin{itemize}
			\item[(i)] $\text{YT}^{^{\lambda}}_{_{(2,4,\ldots,4n-4,4n-2,4n)}}$
			\item[(ii)] $\text{YT}^{^{\lambda}}_{_{(2,4,\ldots,4n-4,4n-1,4n)}}$
			\item[(iii)] $\text{YT}^{^{\lambda}}_{_{(2,4,\ldots,4n-4,4n-2,4n+1)}}$
			\item[(iv)] $\text{YT}^{^{\lambda}}_{_{(2,4,\ldots,4n-4,4n-1,4n+1)}}$
			\item[(v)] $\text{YT}^{^{\lambda}}_{_{(2,4,\ldots,4n-4,4n-2,4n+2)}}$
			\item[(vi)] $\text{YT}^{^{\lambda}}_{_{(2,4,\ldots,4n-4,4n-1,4n+2)}}.$
		\end{itemize}
		
		Case (i): Assume that $row_{2k}=\text{YT}^{^{\lambda}}_{_{(2,4,\ldots,4n-4,4n-2,4n)}}$.
		
		 Then \begin{equation*}
			e_{i,j}=\left\{\begin{array}{lr}
				4n-2 & \text{for $k+1 \leq i \leq 2k$ and $j=2n-1$}\\
				4n-1 & \text{for $1 \leq i \leq k$ and $j=2n$}\\
				4n & \text{for $k+1 \leq i \leq 2k$ and $j=2n$}\\
				4n+1 & \text{for $1 \leq i \leq k$ and $j=2n+1.$}\\
			\end{array}\right.
		\end{equation*} Therefore, $row_1=\text{YT}^{^{\lambda}}_{_{(1,3,\ldots,4n-3,4n-1,4n+1,4n+2)}}$ and hence, $row_{1}, row_{2k}$ together give a factor $X_1$ of $p_{\Gamma}$.
		
		Case (ii): Assume that $row_{2k}=\text{YT}^{^{\lambda}}_{_{(2,4,\ldots,4n-4,4n-1,4n)}}$.
		
		 Then $e_{i,2n}=4n$ for all $k+1 \leq i \leq 2k$ and $e_{1,2n}=4n-2$ or $4n-1.$ If $e_{1,2n}=4n-1,$ then all $4n-2$ appear in the $(2n-1)$-th column. Further, since $e_{2k,2n-1}=4n-1,$ $n_{4n-2,2n-1} \leq k-1,$ which is a contradiction to \cref{eq12.2}. Thus, $e_{1,2n}=4n-2.$ Therefore, $row_{1}=\text{YT}^{^{\lambda}}_{_{(1,3,\ldots,4n-3,4n-2,4n+1,4n+2)}}$ and hence, $row_{1}, row_{2k}$ together give a factor $X_2$ of $p_{\Gamma}$.  
		
		Case (iii): Assume that $row_{2k}=\text{YT}^{^{\lambda}}_{_{(2,4,\ldots,4n-4,4n-2,4n+1)}}$. 
		
		Then $e_{i,2n-1}=4n-2$ for all $k+1 \leq i \leq 2k$ and $e_{i,2n}=4n-1$ for all $1 \leq i \leq k.$ Since $e_{2k,2n}=4n+1,$ we have $n_{4n,2n} \leq k-1$. Thus, $n_{4n,2n+1} \geq 1$. Thus, we have $e_{1,2n+1}=4n$. Further, since $\Gamma$ satisfies the property that number of entries greater than $4n+2$ in each row is even, we have $row_{1}=\text{YT}^{^{\lambda}}_{_{(1,3,\ldots,4n-3,4n-1,4n,4n+1)}}$ or $\text{YT}^{^{\lambda}}_{_{(1,3,\ldots,4n-3,4n-1,4n,4n+2)}}$ (see \cref{subsection2.1}). 
		
		Claim: $row_{1} \neq \text{YT}^{^{\lambda}}_{_{(1,3,\ldots,4n-3,4n-1,4n,4n+1)}}.$ 
		
		Assume that $row_1 = \text{YT}^{^{\lambda}}_{_{(1,3,\ldots,4n-3,4n-1,4n,4n+1)}}.$ Let $k'$ denote the number of rows in $\Gamma$ those contain exactly $2n+2$ entries less than or equal to $4n+2.$ Then remaining $2k-k'$ number of rows in $\Gamma$ contains exactly $2n$ entries less than or equal to $4n+2.$ Since by \cref{eq12.2}, $c_{\Gamma}(t)=k$ for all $1 \leq t \leq 4n+2,$ we have $(2n+2)k'+2n(2k-k')=(4n+2)k.$ Thus, we have $k'=k.$ Thus, $row_i$ ($1 \leq i \leq k$) contains exactly $2n$ entries less than or equal to $4n+2$ and $row_{i}$ ($k+1 \leq i \leq 2k$) contains exactly $2n+2$ entries less than or equal to $4n+2$. Further, since $e_{2k,2n}=4n+1,$ all $4n+2$ appear in the $(2n+2)$-th column and hence, $n_{4n+2,2n+2} \leq k-1$, which is a contradiction to \cref{eq12.2}.

		On the other hand, if $row_1=\text{YT}^{^{\lambda}}_{_{(1,3,\ldots,4n-3,4n-1,4n,4n+2)}},$ then $row_{1}, row_{2k}$ together give a factor $X_3$ of $p_{\Gamma}.$	
		
		Case (iv): Assume that $row_{2k}=\text{YT}^{^{\lambda}}_{_{(2,4,\ldots,4n-4,4n-1,4n+1)}}$. 
		
		Since $e_{2k,2n-1}=4n-1,$ we have $n_{4n-2,2n-1} \leq k-1.$ Hence, $n_{4n-2,2n} \geq 1.$ Therefore, $e_{1,2n}=4n-2$. Now we claim that $e_{1,2n+1} \leq 4n.$ If $e_{1,2n+1} \geq 4n+1,$ then all integers less than or equal to $4n$ appear in the first $2n$ columns. Since $e_{2k,2n}=4n+1,$ we have $\sum_{t=1}^{4n}\sum_{j=1}^{2n}n_{t,j} \leq 4nk-1,$ which is a contradiction to \cref{eq12.2}. Therefore, $e_{1,2n+1} \leq 4n.$ Hence, $row_1$ is one of the following:
		\begin{itemize}
			\item $\text{YT}^{^{\lambda}}_{_{(1,3,\ldots,4n-3,4n-2,4n-1,4n,4n+1,4n+2)}}$
			\item $\text{YT}^{^{\lambda}}_{_{(1,3,\ldots,4n-3,4n-2,4n-1,4n)}}$
			\item $\text{YT}^{^{\lambda}}_{_{(1,3,\ldots,4n-3,4n-2,4n-1,4n+1)}}$
			\item $\text{YT}^{^{\lambda}}_{_{(1,3,\ldots,4n-3,4n-2,4n-1,4n+2)}}$
			\item $\text{YT}^{^{\lambda}}_{_{(1,3,\ldots,4n-3,4n-2,4n,4n+1)}}$
			\item $ \text{YT}^{^{\lambda}}_{_{(1,3,\ldots,4n-3,4n-2,4n,4n+2)}}.$
		\end{itemize}
		
		If $row_1=
		\text{YT}^{^{\lambda}}_{_{(1,3,\ldots,4n-3,4n-2,4n-1,4n)}},$  $\text{YT}^{^{\lambda}}_{_{(1,3,\ldots,4n-3,4n-2,4n-1,4n+1)}},$ or $\text{YT}^{^{\lambda}}_{_{(1,3,\ldots,4n-3,4n-2,4n,4n+1)}},$
		then by using the argument as in the third paragraph (line no. 7--12) of case (iii), we see that $row_i$ ($1 \leq i \leq k$) contains exactly $2n+2$ entries less than or equal to $4n+2$ and $row_{i}$ ($k+1 \leq i \leq 2k$) contains exactly $2n$ entries less than or equal to $4n+2$. Further, since $e_{2k,2n}=4n+1,$ all $4n+2$ appear in the $(2n+2)$-th column. Since $e_{1,2n+2}=4n$ or $4n+1,$ we have $n_{4n+2,2n+2} \leq k-1$, which is a contradiction to \cref{eq12.2}.
		
		If $row_1 = \text{YT}^{^{\lambda}}_{_{(1,3,\ldots,4n-3,4n-2,4n-1,4n,4n+1,4n+2)}},$ then for $k+1 \leq i \leq 2k-1,$ we have $e_{i,2n-1}=4n-2$ or $4n-1.$ Further, since $e_{1,2n+1}=4n-1,$ we have $n_{4n-2,2n-1}+n_{4n-1,2n-1}+n_{4n-1,2n+1} \geq k+1.$ Thus, $n_{4n-2,2n}+n_{4n-1,2n} \leq k-1.$ Thus, $e_{k,2n} = 4n.$ Hence, $row_{k}=\text{YT}^{^{\lambda}}_{_{(1,3,\ldots,4n-3,4n,4n+1,4n+2)}}$ or $\text{YT}^{^{\lambda}}_{_{(1,3,\ldots,4n-3,4n)}}.$ If $row_{k}=\text{YT}^{^{\lambda}}_{_{(1,3,\ldots,4n-3,4n,4n+1,4n+2)}},$ then $\sum_{j=2n+1}^{4n+2}\sum_{t=4n+3}^{8n+4}n_{t,j} \leq (4n+2)k-2,$ which is a contradiction to \cref{eq12.2}. If $row_{k}=\text{YT}^{^{\lambda}}_{_{(1,3,\ldots,4n-3,4n)}},$ then $e_{i,2n+1}=4n+3$ for all $k \leq i \leq 2k,$ which is a contradiction to \cref{eq12.2}.

		If $row_{1}=\text{YT}^{^{\lambda}}_{_{(1,3,\ldots,4n-3,4n-2,4n,4n+2)}},$ then hence, $row_{1}, row_{2k}$ together give a factor $X_4$ of $p_{\Gamma}$.
		
		If $row_{1}=\text{YT}^{^{\lambda}}_{_{(1,3,\ldots,4n-3,4n-2,4n-1,4n+2)}},$ then by using the argument as in the third paragraph (line no. 7--12) of case (iii), we see that $row_i$ ($1 \leq i \leq k$) contains exactly $2n+2$ entries less than or equal to $4n+2$ and $row_{i}$ ($k+1 \leq i \leq 2k$) contains exactly $2n$ entries less than or equal to $4n+2$. Further, since $e_{1,2n+2}=4n+2,$ we have $e_{i,2n+2}=4n+2$ for all $1 \leq i \leq k.$ Since $n_{4n-2,2n-1}+n_{4n-1,2n-1}+n_{4n-1,2n+1} \geq k+1,$ we have $n_{4n-2,2n}+n_{4n-1,2n} \leq k-1.$ Hence, $e_{k,2n} \geq 4n.$ Further, since $e_{2k,2n}=4n+1,$ we have $e_{k,2n}=4n.$ Thus, $row_{k}=\text{YT}^{^{\lambda}}_{_{(1,3,\ldots,4n-3,4n,4n+1,4n+2)}}.$ Since $e_{k,2n+1}=4n+1$ we have $n_{4n+1,2n} \leq k-1.$ Thus, $e_{k+1,2n}=4n.$ Therefore, $row_{k+1}=\text{YT}^{^{\lambda}}_{_{(2,4,\ldots,4n-4,4n-2,4n)}}$ and hence, $row_1, row_k, row_{k+1}, row_{2k}$ together give a factor $Y_1$ of $p_{\Gamma}.$ 
		
		Case (v): Assume that $row_{2k}=\text{YT}^{^{\lambda}}_{_{(2,4,\ldots,4n-4,4n-2,4n+2)}}$. Then $e_{i,2n-1}=4n-2$ for all $k+1 \leq i \leq 2k$ and $e_{i,2n}=4n-1$ for all $1 \leq i \leq k.$  Further, since $e_{2k,2n}=4n+2,$ we have $n_{4n,2n} \leq k-1.$ Thus, $n_{4n,2n+1} \geq 1.$ Hence, $e_{1,2n+1}=4n.$ Therefore, $e_{1,2n+2}=4n+1$ or $4n+2.$ 
		
		Claim: $e_{1,2n+2} \neq 4n+2.$
		
		Assume that $e_{1,2n+2} = 4n+2.$ By using the argument as in the third paragraph (line no. 7--12) of case (iii), we see that $row_i$ ($1 \leq i \leq k$) contains exactly $2n+2$ entries less than or equal to $4n+2$ and $row_{i}$ ($k+1 \leq i \leq 2k$) contains exactly $2n$ entries less than or equal to $4n+2$. Further, since $e_{2k,2n}=4n+2,$ we have $n_{4n+2,2n+2}+n_{4n+2,2n} \geq k+1,$ which is a contradiction to \cref{eq12.2}.
		
		On the other hand, if $e_{1,2n+2}=4n+1,$ then $row_1=\text{YT}^{^{\lambda}}_{_{(1,3,\ldots,4n-3,4n-1,4n,4n+1)}}.$ Therefore, $row_1, row_{2k}$ together give a factor $X_5$ of $p_{\Gamma}$.
		
		Case (vi): Assume that $row_{2k}=\text{YT}^{^{\lambda}}_{_{(2,4,\ldots,4n-4,4n-1,4n+2)}}$. 
		
		Then $e_{1,2n}=4n-2.$ Since $e_{2k,2n+1}=4n+4,$ we have $e_{1,2n+2} \leq 4n+1.$ Indeed, if $e_{1,2n+2}>4n+2,$ then $\sum_{j=2n+1}^{4n+2}\sum_{t=4n+3}^{8n+4}n_{t,j} \geq (4n+2)k+1,$ which is a contradiction to \cref{eq12.2}. If $e_{1,2n+2}=4n+2,$ then by using the argument as in the third paragraph (line no. 7--12) of case (iii), we see that $row_i$ ($1 \leq i \leq k$) contains exactly $2n+2$ entries less than or equal to $4n+2$ and $row_{i}$ ($k+1 \leq i \leq 2k$) contains exactly $2n$ entries less than or equal to $4n+2$. Thus, $e_{i,2n+2}=4n+2,$ for all $1 \leq i \leq k.$ Further, since $e_{2k,2n}=4n+2,$ we have $n_{4n+2,2n}+n_{4n+2,2n+2} \geq k+1,$ which is a contradiction to \cref{eq12.2}. Therefore, $row_{1}$ is one of the following:
		\begin{itemize}
			\item  $ \text{YT}^{^{\lambda}}_{_{(1,3,\ldots,4n-3,4n-2,4n-1,4n+1)}}$
			\item $ \text{YT}^{^{\lambda}}_{_{(1,3,\ldots,4n-3,4n-2,4n,4n+1)}}$
			\item $\text{YT}^{^{\lambda}}_{_{(1,3,\ldots,4n-3,4n-2,4n-1,4n,4n+1,4n+2)}}$ 
			\item $\text{YT}^{^{\lambda}}_{_{(1,3,\ldots,4n-3,4n-2,4n-1,4n)}}.$
		\end{itemize} 
		
		Subcase (i): Assume that $row_1=\text{YT}^{^{\lambda}}_{_{(1,3,\ldots,4n-3,4n-2,4n-1,4n+1)}}.$ 
		
		Since $n_{4n-1,2n+1} \geq 1,$ we have $e_{k+1,2n-1}=4n-2.$ Further, since $n_{4n-2,2n-1}+n_{4n-1,2n-1}+n_{4n-1,2n+1} \geq k+1,$ we have $n_{4n-2,2n}+n_{4n-1,2n} \leq k-1.$ Thus, $e_{k,2n} \geq 4n.$ Also, by using the argument as in the third paragraph (line no. 7--12) of case (iii), we see that $row_i$ ($1 \leq i \leq k$) contains exactly $2n+2$ entries less than or equal to $4n+2$ and $row_{i}$ ($k+1 \leq i \leq 2k$) contains exactly $2n$ entries less than or equal to $4n+2$. Therefore, $row_{k}=\text{YT}^{^{\lambda}}_{_{(1,3,\ldots,4n-3,4n,4n+1,4n+2)}}.$ Hence, $e_{k+1,2n} \geq 4n.$ If $e_{k+1,2n} \geq 4n+1,$ then $\sum_{t=4n+1}^{4n+2}\sum_{j=2n}^{2n+2}n_{t,j} \geq 2k+1,$ which is a contradiction to \cref{eq12.2}. Hence, $e_{k+1,2n}=4n$ and therefore, $row_1, row_{k}, row_{k+1}, row_{2k}$ together give a factor $Y_2$ of $p_{\Gamma}.$ 
		
		Subcase (ii): If $row_1=\text{YT}^{^{\lambda}}_{_{(1,3,\ldots,4n-3,4n-2,4n,4n+1)}},$ then $row_1, row_{2k}$ together give a factor $X_6$ of $p_{\Gamma}.$
		
		Subcase (iii): Assume that $row_{1}=\text{YT}^{^{\lambda}}_{_{(1,3,\ldots,4n-3,4n-2,4n-1,4n,4n+1,4n+2)}}.$
		
		 Since $e_{1,2n+1}=4n-1,$ we have $e_{k+1,2n-1}=4n-2.$ Further, since $n_{4n-2,2n-1}+n_{4n-1,2n-1}+n_{4n-1,2n+1} \geq k+1,$ we have $n_{4n-2,2n}+n_{4n-1,2n} \leq k-1.$ Hence, $e_{k,2n} \geq 4n.$ Next we show that $e_{k,2n}=4n.$ If $e_{k,2n}=4n+2,$ then we have $n_{4n+2,2n} \geq k+1,$ which is a contradiction to \cref{eq12.2}. If $e_{k,2n}=4n+1,$ then $row_{k}=\text{YT}^{^{\lambda}}_{_{(1,3,\ldots,4n-3,4n+1)}}.$ Thus, $e_{i,2n+2}=4n+5$ for all $k \leq i \leq 2k,$ which is a contradiction to \cref{eq12.2}. Thus, $e_{k,2n}=4n$. Therefore, $row_k$ is either $\text{YT}^{^{\lambda}}_{_{(1,3,\ldots,4n-3,4n,4n+1,4n+2)}}$ or $\text{YT}^{^{\lambda}}_{_{(1,3,\ldots,4n-3,4n)}}.$ If $row_k=\text{YT}^{^{\lambda}}_{_{(1,3,\ldots,4n-3,4n,4n+1,4n+2)}},$ then  $\sum_{j=2n+1}^{4n+2}\sum_{t=4n+3}^{8n+4}n_{t,j} \leq (4n+2)k-2,$ which is a contradiction to \cref{eq12.2}.
		
		Now assume that $row_{k}=\text{YT}^{^{\lambda}}_{_{(1,3,\ldots,4n-3,4n)}}.$ 
		
		Since $e_{1,2n+4}=4n+2,$ we have $e_{k+1,2n}=4n$ or $4n+1.$ If $e_{k+1,2n}=4n+1,$ then $row_{k+1}=\text{YT}^{^{\lambda}}_{_{(2,4,\ldots,4n-4,4n-2,4n+1)}}.$ Therefore, $row_1, row_{k}, row_{k+1}, row_{2k}$ together give a factor $Y_3$ of $p_{\Gamma}.$ 
		
		If $e_{k+1,2n}=4n,$ then $row_{k+1}=\text{YT}^{^{\lambda}}_{_{(2,4,\ldots,4n-4,4n-2,4n)}}$. Since $row_k$ does not contain $4n+1,$ there exists $l$ for some $k+2 \leq l \leq 2k-1$ such that $row_{l}$ contains $4n+1.$ Hence, $row_l$ is either $\text{YT}^{^{\lambda}}_{_{(2,4,\ldots,4n-4,4n-2,4n+1)}}$ or $\text{YT}^{^{\lambda}}_{_{(2,4,\ldots,4n-4,4n-1,4n+1)}}$. If $row_l=\text{YT}^{^{\lambda}}_{_{(2,4,\ldots,4n-4,4n-2,4n+1)}},$ then $row_1, row_{k}, row_{l}, row_{2k}$ together give a factor $Y_3$ of $p_{\Gamma}$. 
		
		Consider $row_l=\text{YT}^{^{\lambda}}_{_{(2,4,\ldots,4n-4,4n-1,4n+1)}}.$ Then we study $row_{\frac{k}{2}+1}.$ 
		
		Assume that $k$ is even. 
		
		If $e_{\frac{k}{2}+1,2n+1}=4n-1,$ then $\sum_{t=4n-2}^{4n-1}(n_{t,2n-1}+n_{t,2n}+n_{t,2n+1}) \geq 2k+2,$ which is a contradiction to \cref{eq12.2}. If $e_{\frac{k}{2}+1,2n+1} \geq 4n+3,$ then $e_{\frac{k}{2}+1,2n+1} = 4n+3$ and $e_{\frac{k}{2}+1,2n+2} = 4n+4.$ Thus, $\sum_{t=4n+3}^{4n+4}\sum_{j=2n+1}^{2n+2}n_{t,j}\geq 2k+1,$ which is a contradiction to \cref{eq12.2}. Hence, $row_{\frac{k}{2}+1}$ is in the following set of rows.
		\begin{itemize}
			\item $\text{YT}^{^{\lambda}}_{_{(1,3,\ldots,4n-3,4n-2,4n,4n+1)}}$
			\item $\text{YT}^{^{\lambda}}_{_{(1,3,\ldots,4n-3,4n-2,4n,4n+2)}}$
			\item $\text{YT}^{^{\lambda}}_{_{(1,3,\ldots,4n-3,4n-2,4n+1,4n+2)}}$ 
			\item $\text{YT}^{^{\lambda}}_{_{(1,3,\ldots,4n-3,4n-1,4n,4n+1)}}$
			\item $\text{YT}^{^{\lambda}}_{_{(1,3,\ldots,4n-3,4n-1,4n,4n+2)}}$
			\item $\text{YT}^{^{\lambda}}_{_{(1,3,\ldots,4n-3,4n-1,4n+1,4n+2)}}$
			\item $\text{YT}^{^{\lambda}}_{_{(1,3,\ldots,4n-3,4n,4n+1,4n+2)}}.$
		\end{itemize}
		
		If $row_{\frac{k}{2}+1}=\text{YT}^{^{\lambda}}_{_{(1,3,\ldots,4n-3,4n-2,4n,4n+1)}},$ then $row_{\frac{k}{2}+1}, row_{2k}$ give a factor $X_6$ of $p_{\Gamma}.$
		
		If $row_{\frac{k}{2}+1}=\text{YT}^{^{\lambda}}_{_{(1,3,\ldots,4n-3,4n-2,4n,4n+2)}},$ then $row_{\frac{k}{2}+1}, row_{l}$ give a factor $X_4$ of $p_{\Gamma}.$ 
		
		If $row_{\frac{k}{2}+1}=\text{YT}^{^{\lambda}}_{_{(1,3,\ldots,4n-3,4n-2,4n+1,4n+2)}},$ then $row_{i} (i=1,\frac{k}{2}+1,k,k+1,l,2k)$ give a factor $Z_2$ of $p_{\Gamma}.$ 
		
		If $row_{\frac{k}{2}+1}=\text{YT}^{^{\lambda}}_{_{(1,3,\ldots,4n-3,4n-1,4n,4n+1)}}$ or $\text{YT}^{^{\lambda}}_{_{(1,3,\ldots,4n-3,4n-1,4n,4n+2)}},$ then $n_{4n-2,2n} \leq \frac{k}{2}$. Thus, $n_{4n-2,2n-1} \geq \frac{k}{2}$. Further, since $row_l=\text{YT}^{^{\lambda}}_{_{(2,4,\ldots,4n-4,4n-1,4n+1)}},$ we have $row_{\frac{3k}{2}}=\text{YT}^{^{\lambda}}_{_{(2,4,\ldots,4n-4,4n-2,4n)}}$. Therefore,  for $1 \leq i \leq \frac{3k}{2},$ $e_{i,2n}=4n-2,4n-1$ or $4n$ and for $1 \leq i \leq \frac{k}{2}+1,$ $e_{i,2n+1}=4n-1$ or $4n.$ Hence, $\sum_{j=2n-1}^{2n+1}\sum_{t=4n-2}^{4n}n_{t,j}\geq k+\frac{3k}{2}+\frac{k}{2}+1=3k+1$, which is a contradiction to \cref{eq12.2}. 
		
		If $row_{\frac{k}{2}+1}=\text{YT}^{^{\lambda}}_{_{(1,3,\ldots,4n-3,4n-1,4n+1,4n+2)}},$ then $row_{\frac{k}{2}+1}, row_{k+1}$ give a factor $X_1$ of $p_{\Gamma}.$ 
		
		If $row_{\frac{k}{2}+1}=\text{YT}^{^{\lambda}}_{_{(1,3,\ldots,4n-3,4n,4n+1,4n+2)}},$ then for $1 \leq i \leq \frac{k}{2},$ $e_{i,2}=4n-2$ and $e_{i,3}=4n-1.$  Thus, $row_i=\text{YT}^{^{\lambda}}_{_{(1,3,\ldots,4n-2,4n)}}$ for all $k+1 \leq i \leq \frac{3k}{2}.$ Hence, $e_{i,2}=4n$ for all $\frac{k}{2}+1 \leq i \leq \frac{3k}{2}$. Since $e_{1,4}=4n,$ we have $n_{4n,2}+n_{4n,4} \geq k+1,$ which is a contradiction to \cref{eq12.2}.

		Similarly, if $k+1$ is even, then the proof is similar.  
		
		Subcase (iv): If $row_{1}=\text{YT}^{^{\lambda}}_{_{(1,3,\ldots,4n-3,4n-2,4n-1,4n)}},$ then the proof is similar to subcase (iii). 
		
	\end{proof}
	
	\begin{lemma}\label{lem5.2}
		The following relations in $X_i$'s $(1 \leq i \leq 6),$ $Y_j$'s ($1 \leq j \leq 4$) hold in $R_2.$
		\begin{itemize}
			\item[(i)] $X_2X_3-X_1X_4+Y_1-Y_3=0.$
			\item[(ii)] $X_4X_5-X_3X_6+Y_4-Y_3=0.$
			\item[(iii)] $X_2X_5-X_1X_6+Y_2-Y_3=0.$
		\end{itemize}
	\end{lemma}
	\begin{proof}
		$(i):$ 
		Note that
		$$X_2X_3=\text{YT}^{^{4\lambda}}_{_{(1,3,\ldots,4n-3,4n-2,4n+1,4n+2),(1,3,\ldots,4n-3,4n-1,4n,4n+2),(2,4,\ldots,4n-4,4n-1,4n),(2,4,\ldots,4n-4,4n-2,4n+1)}}.$$
		
		Now we see that $\text{YT}^{^{2\lambda}}_{_{(1,\ldots,4n-3,4n-2,4n+1,4n+2),(1,\ldots,4n-3,4n-1,4n,4n+2)}}$ is a nonstandard monomial. We apply \cref{thm2.3}, \cref{rmk2.4} and \cref{rmk2.5} to express the above non-standard monomial as a linear combination of standard monomials.
		
		Let $\underline{i}_1=(1,3,\ldots,4n-3,4n-2,4n+1,4n+2,4n+5,4n+6,4n+9,\ldots,8n+3)$ and $\underline{i}_2=(1,3,\ldots,4n-3,4n-1,4n,4n+2,4n+4,4n+7,4n+9,\ldots,8n+3).$ Then $\underline{i}_1(B)=(2,4,\ldots,4n-4,4n-1,4n)$ and $\underline{i}_2(B)=(2,4,\ldots,4n-4,4n-2,4n+1)$ are two subsets of $\{1,2,\ldots,4n+2\}.$
		
		In order to apply \cref{thm2.3} set $I_1=\{2,4,\ldots,4n-4,4n-1\}$ and $I_2=\{2,4,\ldots,4n-4,4n-2,4n,4n+1\}.$ Note that $I_1\Delta I_2=\{4n-2,4n-1,4n,4n+1\}.$ Then we have 
		
		$P(I_1 \Delta \{4n-2\})\cdot P(I_2\Delta \{4n-2\})-P(I_1 \Delta \{4n-1\})\cdot P(I_2 \Delta \{4n-1\})+P(I_1 \Delta \{4n\})\cdot P(I_2 \Delta \{4n\})-P(I_1 \Delta \{4n+1\})\cdot P(I_2 \Delta \{4n+1\})=0.$
		
		By using \cref{rmk2.5} we have the following:
		
		$P(I_1 \Delta \{4n-2\})=q_{\{1,3,\ldots,4n-3,4n,4n+1,4n+2,4n+6,4n+7,4n+9,\ldots,8n+4\}}$
		
		$P(I_2 \Delta \{4n-2\})=q_{\{1,3,\ldots,4n-3,4n-2,4n-1,4n+2,4n+4,4n+5,4n+9,\ldots,8n+4\}}$
		
		$P(I_1 \Delta \{4n-1\})=q_{\{1,3,\ldots,4n-3,4n-2,4n-1,4n,4n+1,4n+2,4n+9,\ldots,8n+4\}}$
		
		$P(I_2 \Delta \{4n-1\})=q_{\{1,3,\ldots,4n-3,4n+2,4n+4,4n+5,4n+6,4n+7,4n+9,\ldots,8n+4\}}$
		
		$P(I_1 \Delta \{4n\})=q_{\underline{i}_1}$
		
		$P(I_2 \Delta \{4n\})=q_{\underline{i}_2}$
		
		$P(I_1 \Delta \{4n+1\})=q_{\{1,3,\ldots,4n-3,4n-2,4n,4n+2,4n+4,4n+6,4n+9,\ldots,8n+4\}}$
		
		$P(I_2 \Delta \{4n+1\})=q_{\{1,3,\ldots,4n-3,4n-1,4n+1,4n+2,4n+5,4n+7,4n+9,\ldots,8n+4\}}.$
		
		Therefore, we have the following straightening laws in $G/P^{\alpha_{4n+2}}$
		
		$\begin{matrix}\text{YT}^{^{2\lambda}}_{_{(1,3,\ldots,4n-3,4n-2,4n+1,4n+2),(1,3,\ldots,4n-3,4n-1,4n,4n+2)}}\\
			=\text{YT}^{^{2\lambda}}_{_{(1,3,\ldots,4n-3,4n-2,4n,4n+2),(1,3,\ldots,4n-3,4n-1,4n+1,4n+2)}}\\
			-\text{YT}^{^{2\lambda}}_{_{(1,3,\ldots,4n-3,4n-2,4n-1,4n+2),(1,3,\ldots,4n-3,4n,4n+1,4n+2)}}\\
			+\text{YT}^{^{2\lambda}}_{_{(1,3,\ldots,4n-3,4n-2,4n-1,4n,4n+1,4n+2),(1,3,\ldots,4n-3,4n+2)}}.
		\end{matrix}$\\
		
		Similarly, by using \cref{thm2.3}, \cref{rmk2.4} and \cref{rmk2.5} we have
		
		$\begin{matrix}\text{YT}^{^{2\lambda}}_{_{(2,4,\ldots,4n-4,4n-2,4n+1),(2,4,\ldots,4n-4,4n-1,4n)}}
			&=\text{YT}^{^{2\lambda}}_{_{(2,4,\ldots,4n-4,4n-2,4n),(2,4,\ldots,4n-4,4n-1,4n+1)}}\\
			&-\text{YT}^{^{2\lambda}}_{_{(2,4,\ldots,4n-4,4n-2,4n-1),(2,4,\ldots,4n-4,4n,4n+1)}}\\
			&+\text{YT}^{^{2\lambda}}_{_{(2,4,\ldots,4n-4,4n-2,4n-1,4n,4n+1),(2,4,\ldots,4n-4)}}.
		\end{matrix}$
		
		Since we are working in the Schubert variety $X(v_6)$, above straightening law becomes
		\begin{equation}\label{eq3.5}
			\text{YT}^{^{2\lambda}}_{_{(2,4,\ldots,4n-4,4n-2,4n+1),(2,4,\ldots,4n-4,4n-1,4n)}}=\text{YT}^{^{2\lambda}}_{_{(2,4,\ldots,4n-4,4n-2,4n),(2,4,\ldots,4n-4,4n-1,4n+1)}}.
		\end{equation}

		Therefore, by using the above straightening laws we have
		
		$\begin{matrix}
			X_2X_3&=&\text{YT}^{^{4\lambda}}_{_{(1,3,\ldots,4n-3,4n-2,4n,4n+2),(1,3,\ldots,4n-3,4n-1,4n+1,4n+2),(2,4,\ldots,4n-4,4n-2,4n),(2,4,\ldots,4n-4,4n-1,4n+1)}}\\
			&-&\text{YT}^{^{4\lambda}}_{_{(1,3,\ldots,4n-3,4n-2,4n-1,4n+2),(1,3,\ldots,4n-3,4n,4n+1,4n+2),(2,4,\ldots,4n-4,4n-2,4n),(2,4,\ldots,4n-4,4n-1,4n+1)}}\\
			&+&\text{YT}^{^{4\lambda}}_{_{(1,3,\ldots,4n-3,4n-2,4n-1,4n,4n+1,4n+2),(1,3,\ldots,4n-3,4n+2),(2,4,\ldots,4n-4,4n-2,4n),(2,4,\ldots,4n-4,4n-1,4n+1)}}.
		\end{matrix}$
		
		Since we are working in the Schubert variety $X(v_6),$ by using \cref{thm2.3} and \cref{rmk2.5} we have
		
		$\text{YT}^{^{2\lambda}}_{_{(1,3,\ldots,4n-3,4n+2),(2,4,\ldots,4n-4,4n-2,4n)}}=\text{YT}^{^{2\lambda}}_{_{(1,3,\ldots,4n-3,4n),(2,4,\ldots,4n-4,4n-2,4n+2)}}.$
		
		Therefore,\\
		$\begin{matrix}
			X_2X_3&=\text{YT}^{^{4\lambda}}_{_{(1,3,\ldots,4n-3,4n-2,4n,4n+2),(1,3,\ldots,4n-3,4n-1,4n+1,4n+2),(2,4,\ldots,4n-4,4n-2,4n),(2,4,\ldots,4n-4,4n-1,4n+1)}}\\
			&-\text{YT}^{^{4\lambda}}_{_{(1,3,\ldots,4n-3,4n-2,4n-1,4n+2),(1,3,\ldots,4n-3,4n,4n+1,4n+2),(2,4,\ldots,4n-4,4n-2,4n),(2,4,\ldots,4n-4,4n-1,4n+1)}}\\
			&+\text{YT}^{^{4\lambda}}_{_{(1,3,\ldots,4n-3,4n-2,4n-1,4n,4n+1,4n+2),(1,3,\ldots,4n-3,4n),(2,4,\ldots,4n-4,4n-2,4n+2),(2,4,\ldots,4n-4,4n-1,4n+1)}}.
		\end{matrix}$
		
		We have the following straightening law in $G/P^{\alpha_{4n+2}}.$
		
		$\begin{matrix}\text{YT}^{^{2\lambda}}_{_{(2,4,\ldots,4n-4,4n-2,4n+2),(2,4,\ldots,4n-4,4n-1,4n+1)}}
			&=\text{YT}^{^{2\lambda}}_{_{(2,4,\ldots,4n-4,4n-2,4n+1),(2,4,\ldots,4n-4,4n-1,4n+2)}}\\
			&-\text{YT}^{^{2\lambda}}_{_{(2,4,\ldots,4n-4,4n-2,4n-1),(2,4,\ldots,4n-4,4n+1,4n+2)}}\\
			&+\text{YT}^{^{2\lambda}}_{_{(2,4,\ldots,4n-4,4n-2,4n-1,4n+1,4n+2),(2,4,\ldots,4n-4)}}.
		\end{matrix}$

		Since we are working in the Schubert variety $X(v_6)$, above straightening law becomes
		
		\begin{equation}\label{eq3.3}
			\text{YT}^{^{2\lambda}}_{_{(2,4,\ldots,4n-4,4n-2,4n+2),(2,4,\ldots,4n-4,4n-1,4n+1)}}=\text{YT}^{^{2\lambda}}_{_{(2,4,\ldots,4n-4,4n-2,4n+1),(2,4,\ldots,4n-4,4n-1,4n+2)}}.
		\end{equation}
		
		Therefore, we have \\
		$\begin{matrix}
			X_2X_3&=\text{YT}^{^{2\lambda}}_{_{(1,3,\ldots,4n-3,4n-2,4n,4n+2),(1,3,\ldots,4n-3,4n-1,4n+1,4n+2),(2,4,\ldots,4n-4,4n-2,4n),(2,4,\ldots,4n-4,4n-1,4n+1)}}\\
			&-\text{YT}^{^{2\lambda}}_{_{(1,3,\ldots,4n-3,4n-2,4n-1,4n+2),(1,3,\ldots,4n-3,4n,4n+1,4n+2),(2,4,\ldots,4n-4,4n-2,4n),(2,4,\ldots,4n-4,4n-1,4n+1)}}\\
			&+\text{YT}^{^{2\lambda}}_{_{(1,3,\ldots,4n-3,4n-2,4n-1,4n,4n+1,4n+2),(1,3,\ldots,4n-3,4n),(2,4,\ldots,4n-4,4n-2,4n+1),(2,4,\ldots,4n-4,4n-1,4n+2)}}
		\end{matrix}$\\
		
		$\hspace{1cm}=X_1X_4-Y_1+Y_3.$

		$(ii):$ Note that
		$$X_4X_5=\text{YT}^{^{4\lambda}}_{_{(1,3,\ldots,4n-3,4n-2,4n,4n+2),(1,3,\ldots,4n-3,4n-1,4n,4n+1),(2,4,\ldots,4n-4,4n-2,4n+2),(2,4,\ldots,4n-4,4n-1,4n+1)}}.$$
		\normalsize{By} using \cref{thm2.3} and \cref{rmk2.5}, the following straightening law holds in $G/P^{\alpha_{4n+2}}.$ 
		\begin{equation}\label{eq3.4}
			\begin{matrix}\text{YT}^{^{2\lambda}}_{_{(1,3,\ldots,4n-3,4n-2,4n,4n+2),(1,3,\ldots,4n-3,4n-1,4n,4n+1)}}
				&=\text{YT}^{^{2\lambda}}_{_{(1,3,\ldots,4n-3,4n-2,4n,4n+1),(1,3,\ldots,4n-3,4n-1,4n,4n+2)}}\\
				&-\text{YT}^{^{2\lambda}}_{_{(1,3,\ldots,4n-3,4n-2,4n-1,4n),(1,3,\ldots,4n-3,4n,4n+1,4n+2)}}\\
				&+\text{YT}^{^{2\lambda}}_{_{(1,3,\ldots,4n-3,4n-2,4n-1,4n,4n+1,4n+2),(1,3,\ldots,4n-3,4n)}}.
			\end{matrix}
		\end{equation}
		\normalsize{Therefore,} by using \cref{eq3.3} and \cref{eq3.4} we have
		
		$\begin{matrix}
			X_4X_5&=\text{YT}^{^{4\lambda}}_{_{(1,3,\ldots,4n-3,4n-2,4n,4n+1),(1,3,\ldots,4n-3,4n-1,4n,4n+2),(2,4,\ldots,4n-4,4n-2,4n+1),(2,4,\ldots,4n-4,4n-1,4n+2)}}\\
			&-\text{YT}^{^{4\lambda}}_{_{(1,3,\ldots,4n-3,4n-2,4n-1,4n),(1,3,\ldots,4n-3,4n,4n+1,4n+2),(2,4,\ldots,4n-4,4n-2,4n+1),(2,4,\ldots,4n-4,4n-1,4n+2)}}\\
			&+\text{YT}^{^{4\lambda}}_{_{(1,3,\ldots,4n-3,4n-2,4n-1,4n,4n+1,4n+2),(1,3,\ldots,4n-3,4n),(2,4,\ldots,4n-4,4n-2,4n+1),(2,4,\ldots,4n-4,4n-1,4n+2)}}\\
		\end{matrix}$
		
		$\hspace{1cm}=X_3X_6-Y_4+Y_3.$
		
		$(iii):$ Note that\\
		$X_2X_5=\text{YT}^{^{4\lambda}}_{_{(1,3,\ldots,4n-3,4n-2,4n+1,4n+2),(1,3,\ldots,4n-3,4n-1,4n,4n+1),(2,4,\ldots,4n-4,4n-2,4n+2),(2,4,\ldots,4n-4,4n-1,4n)}}.$
		
		By using \cref{thm2.3} and \cref{rmk2.5}, the following straightening law holds in $G/P^{\alpha_{4n+2}}.$
		
		$\begin{matrix}\text{YT}^{^{2\lambda}}_{_{(1,3,\ldots,4n-3,4n-2,4n+1,4n+2),(1,3,\ldots,4n-3,4n-1,4n,4n+1)}}\\
			=\text{YT}^{^{2\lambda}}_{_{(1,3,\ldots,4n-3,4n-2,4n,4n+1),(1,3,\ldots,4n-3,4n-1,4n+1,4n+2)}}\\
			-\text{YT}^{^{2\lambda}}_{_{(1,3,\ldots,4n-3,4n-2,4n-1,4n+1),(1,3,\ldots,4n-3,4n,4n+1,4n+2)}}\\
			+\text{YT}^{^{2\lambda}}_{_{(1,3,\ldots,4n-3,4n-2,4n-1,4n,4n+1,4n+2),(1,3,\ldots,4n-3,4n+1)}}
		\end{matrix}$
		
		and
		
		$\begin{matrix}\text{YT}^{^{2\lambda}}_{_{(2,4,\ldots,4n-4,4n-2,4n+2),(2,4,\ldots,4n-4,4n-1,4n)}}
			&=\text{YT}^{^{2\lambda}}_{_{(2,4,\ldots,4n-4,4n-2,4n),(2,4,\ldots,4n-4,4n-1,4n+2)}}\\
			&-\text{YT}^{^{2\lambda}}_{_{(2,4,\ldots,4n-4,4n-2,4n-1),(2,4,\ldots,4n-4,4n,4n+2)}}\\
			&+\text{YT}^{^{2\lambda}}_{_{(2,4,\ldots,4n-4,4n-2,4n-1,4n,4n+2),(2,4,\ldots,4n-4)}}.
		\end{matrix}$ 
		
		Since we are working in the Schubert variety $X(v_6)$, the above straightening law becomes 
		
		$\text{YT}^{^{2\lambda}}_{_{(2,4,\ldots,4n-4,4n-2,4n+2),(2,4,\ldots,4n-4,4n-1,4n)}}=\text{YT}^{^{2\lambda}}_{_{(2,4,\ldots,4n-4,4n-2,4n),(2,4,\ldots,4n-4,4n-1,4n+2)}}.$
		
		Therefore, by using the above straightening laws we have
		
		$\begin{matrix}
			X_2X_5&=\text{YT}^{^{4\lambda}}_{_{(1,3,\ldots,4n-3,4n-2,4n,4n+1),(1,3,\ldots,4n-3,4n-1,4n+1,4n+2),(2,4,\ldots,4n-4,4n-2,4n),(2,4,\ldots,4n-4,4n-1,4n+2)}}\\
			&-\text{YT}^{^{4\lambda}}_{_{(1,3,\ldots,4n-3,4n-2,4n-1,4n+1),(1,3,\ldots,4n-3,4n,4n+1,4n+2),(2,4,\ldots,4n-4,4n-2,4n),(2,4,\ldots,4n-4,4n-1,4n+2)}}\\
			&+\text{YT}^{^{4\lambda}}_{_{(1,3,\ldots,4n-3,4n-2,4n-1,4n,4n+1,4n+2),(1,3,\ldots,4n-3,4n+1),(2,4,\ldots,4n-4,4n-2,4n),(2,4,\ldots,4n-4,4n-1,4n+2)}}.\\
		\end{matrix}$
		
		Since we are working in the Schubert variety $X(v_{6})$, by using \cref{thm2.3} and \cref{rmk2.5} we have
		
		$\text{YT}^{^{2\lambda}}_{_{(1,3,\ldots,4n-3,4n+1),(2,4,\ldots,4n-4,4n-2,4n)}}=\text{YT}^{^{2\lambda}}_{_{(1,3,\ldots,4n-3,4n),(2,4,\ldots,4n-4,4n-2,4n+1)}}.$
		
		Therefore,  we have
		
		$\begin{matrix}
			X_2X_5&=\text{YT}^{^{4\lambda}}_{_{(1,3,\ldots,4n-3,4n-2,4n,4n+1),(1,3,\ldots,4n-3,4n-1,4n+1,4n+2),(2,4,\ldots,4n-4,4n-2,4n),(2,4,\ldots,4n-4,4n-1,4n+2)}}\\
			&-\text{YT}^{^{4\lambda}}_{_{(1,3,\ldots,4n-3,4n-2,4n-1,4n+1),(1,3,\ldots,4n-3,4n,4n+1,4n+2),(2,4,\ldots,4n-4,4n-2,4n),(2,4,\ldots,4n-4,4n-1,4n+2)}}\\
			&+\text{YT}^{^{4\lambda}}_{_{(1,3,\ldots,4n-3,4n-2,4n-1,4n,4n+1,4n+2),(1,3,\ldots,4n-3,4n),(2,4,\ldots,4n-4,4n-2,4n+1),(2,4,\ldots,4n-4,4n-1,4n+2)}}
		\end{matrix}$
		
		$\hspace{1cm}=X_1X_6-Y_2+Y_3.$
	\end{proof} 
	\begin{lemma}\label{lem5.3}
		The following relations in $X_i$'s $(1 \leq i \leq 6),$ $Y_j$'s ($1 \leq j \leq 4$) and $Z_1, Z_2$ hold in $R_3.$
		\begin{itemize}
			\item[(i)] $Z_1-X_2Y_4=0.$
			\item[(ii)] $Z_2-X_2Y_3=0.$
		\end{itemize}
	\end{lemma}
	\begin{proof} (i): Note that\\ 
		$X_2Y_4=\text{YT}^{^{4\lambda}}_{_{(1,3,\ldots,4n-3,4n-2,4n-1,4n),(1,3,\ldots,4n-3,4n-2,4n+1,4n+2),(1,3,\ldots,4n-3,4n,4n+1,4n+2),(2,4,\ldots,4n-4,4n-2,4n+1)}}\\
		\text{YT}^{^{2\lambda}}_{_{(2,4,\ldots,4n-4,4n-1,4n),(2,4,\ldots,4n-4,4n-1,4n+2)}}.$
		
		Recall that $\text{YT}^{^{2\lambda}}_{_{(2,4,\ldots,4n-4,4n-2,4n+1),(2,4,\ldots,4n-4,4n-1,4n)}}=\text{YT}^{^{2\lambda}}_{_{(2,4,\ldots,4n-4,4n-2,4n),(2,4,\ldots,4n-4,4n-1,4n+1)}}$ (see \cref{eq3.5}). Therefore,	we have $X_2Y_4=Z_1.$ 
		
		(ii): Note that\\ 
		$X_2Y_3=\text{YT}^{^{4\lambda}}_{_{(1,3,\ldots,4n-3,4n-2,4n-1,4n,4n+1,4n+2),(1,3,\ldots,4n-3,4n-2,4n+1,4n+2),(1,3,\ldots,4n-3,4n),(2,4,\ldots,4n-4,4n-2,4n+1)}}\\
		\text{YT}^{^{2\lambda}}_{_{(2,4,\ldots,4n-4,4n-1,4n),(2,4,\ldots,4n-4,4n-1,4n+2)}}.$
		
		Recall that $\text{YT}^{^{2\lambda}}_{_{(2,4,\ldots,4n-4,4n-2,4n+1),(2,4,\ldots,4n-4,4n-1,4n)}}=\text{YT}^{^{2\lambda}}_{_{(2,4,\ldots,4n-4,4n-2,4n),(2,4,\ldots,4n-4,4n-1,4n+1)}}$ (see \cref{eq3.5}). Therefore, we have $X_2Y_3=Z_2.$
	\end{proof}
	\begin{lemma}\label{lemma5.4}
		$R$ is generated by $X_i,$ $(1 \leq i \leq 6)$ and $Y_3$ as an $\mathbb{C}$-algebra.
	\end{lemma}
	\begin{proof}
		Follows from \cref{lem5.1}, \cref{lem5.2}, and \cref{lem5.3}.
	\end{proof}
	
	Recall that by \cref{thm2.1}, the very ample line bundle $\mathcal{L}(4\lambda)$ descends to a line bundle on the GIT quotient $T \backslash \backslash (G/P^{\alpha_{4n+2}})^{ss}_{T}(\mathcal{L}(4\lambda)).$ We prove that the GIT quotient $T \backslash \backslash (X(v_6))^{ss}_{T}(\mathcal{L}(4\lambda))$ is projectively normal with respect to the descent of the $T$-linearized very ample line bundle $\mathcal{L}(4\lambda).$
	
	\begin{theorem}\label{thm5.1}
		The homogeneous coordinate ring of $T \backslash \backslash (X(v_6))^{ss}_{T}(\mathcal{L}(4\lambda))$ is generated by elements of degree one. 
	\end{theorem}
	\begin{proof}
		Let $f \in H^{0}(X(v_6), \mathcal{L}^{\otimes k}(4\lambda))^{T}=H^{0}(X(v_6), \mathcal{L}^{\otimes 2k}(2\lambda))^{T}.$ Then by \cref{lemma5.4}, we have $$f = \sum {a_{\underline{m},n_1}}X^{\underline{m}}Y_3^{n_1},$$ where $\underline{m}=(m_1,m_2,m_3,m_4,m_5,m_6) \in \mathbb{Z}_{\geq 0}^6$ and $n_1 \in \mathbb{Z}_{\ge 0}$ such that $m_1+m_2+m_3+m_4+m_5+m_6+2n_1=2k,$ $X^{\underline{m}}$ denotes $X_1^{m_1}X_2^{m_2}X_3^{m_3}X_4^{m_4}X_5^{m_5}X_6^{m_6}$ and $a_{\underline{m}, n_1}$'s are non-zero scalars.
		
		Now to prove that the homogeneous coordinate ring of $T \backslash \backslash (X(v_6))^{ss}_{T}(\mathcal{L}(4\lambda))$ is generated by $H^0(X(v_6), \mathcal{L}(4\lambda))^{T}$ as a $\mathbb{C}$-algebra, it is enough to show that for each $f$ as above and each monomial appearing in the expression of $f$ is a product of $k$ elements of $H^0(X(v_6), \mathcal{L}(4\lambda))^{T}.$
		
		Consider the monomial $X^{\underline{m}}Y_{3}^{n_1}$ in the expression of $f.$ Note that $m_1+m_2+m_3+m_4+m_5+m_6$ is an even integer. Thus, $X^{\underline{m}}$ can be written as $\prod_{(i,j)}X_iX_j,$ where the number of pairs $(i,j)$ is $k-n_1$ and repetitions of $X_i$'s are allowed. Therefore, $X^{\underline{m}}$ can be written as product of $k-n_1$ number of monomials in $H^0(X(v_6), \mathcal{L}(4\lambda))^{T}.$ On the other hand, since $Y_3\in H^0(X(v_6), \mathcal{L}(4\lambda))^{T},$ $X^{\underline{m}}Y_{3}^{n_1}$ is a product of $k$ elements of $H^0(X(v_6), \mathcal{L}(4\lambda))^{T}.$
	\end{proof}
	
	\begin{proof}[Proof of \cref{cor3.8}]
		
		Since $X(v_6)$ is normal, the variety ${T} \backslash \backslash (X(v_6))^{ss}_{T}(\mathcal{L}(4\lambda))$ is normal.  Therefore, by using \cref{thm5.1}, it follows that $T \backslash \backslash (X(v_6))^{ss}_{T}(\mathcal{L}(4\lambda))$ is projectively normal with respect to the descent of the $T$-linearized very ample line bundle $\mathcal{L}(4\lambda).$
	\end{proof}

	\begin{corollary}\label{cor5.4}
		Let $v \in W^{P^{\alpha_{4n+2}}}$ be such that $v_1 \leq v\le v_{6}.$ Then the GIT quotient $T \backslash \backslash (X(v))^{ss}_{T}\\(\mathcal{L}(4\lambda))$ is projectively normal  with respect to the descent of the $T$-linearized very ample line bundle $\mathcal{L}(4\lambda).$ 
	\end{corollary}
	
	\begin{proof}
		Since $T$ is linearly reductive, the restriction map $$\phi: H^{0}(X(v_{6}), \mathcal{L}^{\otimes k}(4\lambda))^{T} \longrightarrow H^{0}(X(v), \mathcal{L}^{\otimes k}(4\lambda))^{T}$$ is surjective for all $k \geq 1$.
		
		So, by \cref{cor3.8}, $T \backslash \backslash (X(v))^{ss}_{T}(\mathcal{L}(4\lambda))$ is projectively normal with respect to the descent of the $T$-linearized very ample line bundle $\mathcal{L}(4\lambda).$
	\end{proof}
	
	\begin{proof}[Proof of \cref{prop1.1}] Note that $Y_{3}=0$ on $X(v_{i})$ for all $1 \leq i \leq 5.$

		(i):  Since $T$ is linearly reductive, the restriction map $$\phi: H^{0}(X(v_6), \mathcal{L}^{\otimes k}(4\lambda))^{T} \\ \longrightarrow H^{0}(X(v_{1}), \mathcal{L}^{\otimes k}(4\lambda))^{T}$$ is surjective for all $k \geq 1$. Thus, by \cref{lemma5.4}, any standard monomial in $H^{0}(X(v_1),\mathcal{L}^{\otimes k}(4\lambda))^{T}$ is of the form $X_1^{2k}.$ Hence, $T\backslash\backslash(X(v_1))^{ss}_{T}(\mathcal{L}(4\lambda))=Proj(\mathbb{C}[X_1^2]).$ 
		
		(ii):  Since $T$ is linearly reductive, the restriction map $$\phi: H^{0}(X(v_6), \mathcal{L}^{\otimes k}(4\lambda))^{T} \longrightarrow H^{0}(X(v_{2}), \mathcal{L}^{\otimes k}(4\lambda))^{T}$$ is surjective for all $k \geq 1$. Thus, by \cref{lemma5.4}, $X_1^2, X_1X_2, X_2^2$ are standard monomials in $H^{0}(X(v_2),\mathcal{L}(4\lambda))^{T}.$ 
		
		Note that by \cref{cor5.4}, $T \backslash \backslash (X(v_2))^{ss}_{T}(\mathcal{L}(4\lambda))$ is projectively normal. Therefore, there is a surjective homomorphism of $\mathbb{C}$-algebras $$\phi: \mathbb{C}[z_0,z_1,z_2] \longrightarrow \oplus_{k \in \mathbb{Z}_{\geq 0}}H^{0}(X(v_2), \mathcal{L}^{\otimes k}(4\lambda))^{T},$$ defined by
		$$(z_0, z_1, z_2) \mapsto (X_1^2, X_1X_2, X_2^2).$$ 
		
		Let $I$ be the ideal of $\mathbb{C}[z_0,z_1,z_2]$ generated by the relation
		\begin{equation}\label{5.1}
			z_0z_2=z_1^2.
		\end{equation} 
		
		Clearly, $I \subseteq \ker\phi$ and $\phi$ induces a homomorphism of $\mathbb{C}$-algebras $$\tilde{\phi}: \mathbb{C}[z_0,z_1,z_2]/I \longrightarrow \oplus_{k \in \mathbb{Z}_{\geq 0}}H^{0}(X(v_2), \mathcal{L}^{\otimes k}(4\lambda))^{T}.$$ Now, we show that $\tilde{\phi}$ is an isomorphism. To complete the proof we use  \eqref{5.1} as a reduction system. The process is of replacing a monomial $M$ in $z_i$'s which is divisible by a term $L_i$'s on the left hand side of the reduction rule $L_i=R_i$ by $(M/L_i)R_i,$ where $R_i$ is the right hand side of the reduction rule. We show that the diamond lemma of ring theory holds for this reduction system (see \cite{bergman}). That is any monomial in $z_i$'s reduces after applying the reduction rule to a unique expression in $z_i$'s, in which no term is divisible by a term appearing on the left hand side of the above reduction system.
		
		Since we have only one reduction rule, it is enough to check for $z_0z_1z_2.$ Note that $z_0z_1z_2=(z_0z_2)z_1=z_1^3,$ for which no further reduction is possible. Therefore, $\tilde{\phi}$ is an isomorphism. Hence, the GIT quotient  $T\backslash\backslash(X(v_2))^{ss}_{T}(\mathcal{L}(4\lambda))$ is isomorphic to $Proj(\frac{\mathbb{C}[z_0, z_1, z_2]}{(z_1^2-z_0z_2)}).$ Thus, the GIT quotient $T\backslash\backslash(X(v_2))^{ss}_{T}(\mathcal{L}(\lambda))$ is isomorphic to  $(\mathbb{P}^1, \mathcal{O}_{\mathbb{P}^1}(2))$ as a polarized variety.
		
		Proof of (iii) is similar to the proof of (ii).
		
		(iv):  Since $T$ is linearly reductive, the restriction map $$\phi: H^{0}(X(v_6), \mathcal{L}^{\otimes k}(4\lambda))^{T} \\ \longrightarrow H^{0}(X(v_{4}), \mathcal{L}^{\otimes k}(4\lambda))^{T}$$ is surjective for all $k \geq 1$. Thus, by \cref{lem5.1}, $X_1^2,$ $X_2^2,$ $X_3^2,$ $X_4^2,$ $X_1X_2,$ $X_1X_3,$ $X_1X_4,$ $X_2X_4,$ $X_3X_4,$ $Y_1$ are standard monomials in $H^0(X(v_4), \mathcal{L}(4\lambda))^{T}.$ Further, by \cref{lem5.2}, $X_2X_3=X_1X_4-Y_1.$  
		Note that by \cref{cor5.4}, $T \backslash \backslash (X(v_4))^{ss}_{T}(\mathcal{L}(4\lambda))$ is projectively normal. Let $A=\mathbb{C}[z_1,z_2,z_3,z_4,z_5,z_6,\\z_7,z_8,z_9,z_{10}]$ be the polynomial algebra with variables $z_i$'s ($1 \leq i \leq 10$). Then there is a surjective homomorphism $$\phi: A \rightarrow \oplus_{k \in \mathbb{Z}_{\geq 0}}H^{0}(X(v_4), \mathcal{L}^{\otimes k}(4\lambda))^{T},$$ defined by \begin{equation*} 
			\begin{split} (z_1, z_2, \ldots, z_{10}) \mapsto (X_1^2, X_2^2, X_3^2, X_4^2, X_1X_2, X_1X_3, X_1X_4, X_2X_3, X_2X_4, X_3X_4).
			\end{split}
		\end{equation*} 
		
		Let $I$ be the ideal of $A$ generated by the following relations:
		
		\begin{tabular}{ |p{2.5cm}|p{2.5cm}|p{2.5cm}|p{2.5cm}|p{2.5cm}| }
			\hline
			$z_1z_2=z_5^2$   & $z_1z_3=z_6^2$    & $z_1z_4=z_7^2$ &  $z_1z_8=z_5z_6$ &	$z_1z_9=z_5z_7$\\
			$z_1z_{10}=z_6z_7$  & $z_2z_3=z_8^2$   & $z_2z_4=z_9^2$ & $z_2z_6=z_5z_8$ & $z_2z_7=z_5z_9$ \\
			$z_2z_{10}=z_8z_9$ &  $z_3z_4=z_{10}^2$ &	$z_3z_5=z_6z_8$    & $z_3z_7=z_6z_{10}$ & $z_3z_9=z_8z_{10}$\\
			$z_4z_5=z_7z_9$ &	$z_4z_6=z_7z_{10}$ &   $z_4z_8=z_9z_{10}$ & $z_5z_{10}=z_6z_9$ & $z_6z_{9}=z_7z_8$\\
			\hline
		\end{tabular}.

		Clearly, $I \subseteq \ker\phi$ and $\phi$ induces a homomorphism of $\mathbb{C}$-algebras $$\tilde{\phi}: A/ I \longrightarrow\bigoplus\limits_{k \in \mathbb{Z}_{\geq 0}}H^{0}(X(v_4), \mathcal{L}^{\otimes k}(4\lambda))^T.$$ Now, we show that $\tilde{\phi}$ is an isomorphism. 
		
		To complete the proof we use the above relations as a reduction system. We show that diamond lemma (see \cite{bergman}) holds for this reduction system by looking at the reduction of the minimal overlapping ambiguities.
		
		Note that $z_1z_2z_3=(z_1z_2)z_3=z_3z_5^2 ~(\text{using } z_1z_2=z_5^2)=(z_3z_5)z_5=z_5z_6z_8 ~(\text{using } z_3z_5=z_6z_8).$ 
		
		Again, $z_1z_2z_3=z_1(z_2z_3)=z_1z_8^2 ~(\text{using } z_2z_3=z_8^2)=(z_1z_8)z_8=z_5z_6z_8 ~(\text{using }z_1z_8=z_5z_6).$
		
		Also, $z_1z_2z_3=(z_1z_3)z_2=z_2z_6^2 ~(\text{using } z_1z_3=z_6^2)=(z_2z_6)z_6=z_5z_6z_8 ~(\text{using } z_2z_6=z_5z_8).$
		
		Therefore, $z_1z_2z_3$ reduces to a unique expression.   
		
		Likewise, we can show that the remaining overlapping ambiguities reduce to a unique expression in $z_i$'s after applying different reduction rules. Therefore, $\tilde{\phi}$ is an isomorphism. 
		
		Consider the embedding $$\psi: \mathbb{P}^3 \hookrightarrow \mathbb{P}^9$$ given by $$(x_1, x_2, x_3, x_4) \mapsto (x_1^2, x_2^2, x_3^2, x_4^2, x_1x_2, x_1x_3, x_1x_4, x_2x_3, x_2x_4, x_3x_4).$$ Let $J$ be the homogeneous ideal of $A$ generated by $2 \times 2$ minors of the matrix $$\begin{pmatrix}
			z_1 & z_5 & z_6 & z_7 \\
			z_5 & z_2 & z_8 & z_9 \\
			z_6 & z_8 & z_3 & z_{10} \\
			z_7 & z_9 & z_{10} & z_4 \\
		\end{pmatrix},$$ where $z_i$'s are the homogeneous coordinates of $\mathbb{P}^9.$ Note that $\psi(\mathbb{P}^3)$ is given by $J.$ Further, note that $I \subseteq J$ and $dim (A/I)=dim(A/J)=4$ (computed using Macaulay2 \cite{GS}). Hence, the GIT quotient $T\backslash\backslash(X(v_4))^{ss}_{T}(\mathcal{L}(4\lambda))$ is isomorphic to $(\mathbb{P}^3, \mathcal{O}_{\mathbb{P}^3}(2))$ as a polarized variety.
		
		(v):  Since $T$ is linearly reductive, the restriction map $$\phi: H^{0}(X(v_6), \mathcal{L}^{\otimes k}(4\lambda))^{T} \\ \longrightarrow H^{0}(X(v_{5}), \mathcal{L}^{\otimes k}(4\lambda)^{T}$$ is surjective for all $k \geq 1$. Thus, by \cref{lem5.1}, $X_1^2,$ $X_3^2,$ $X_5^2,$ $X_1X_3,$ $X_1X_5,$ $X_3X_5$ are standard monomials in $H^0(X(v_5), \mathcal{L}(4\lambda))^{T}.$ 		
		Note that by \cref{cor5.4}, $T \backslash \backslash (X(v_5))^{ss}_{T}(\mathcal{L}(4\lambda))$ is projectively normal.  Let $A=\mathbb{C}[z_1,z_2,z_3,z_4,z_5,z_6]$ be the polynomial algebra with variables $z_i$'s. Then there is a surjective homomorphism $$\phi: A \longrightarrow \oplus_{k \in \mathbb{Z}_{\geq 0}}H^{0}(X(v_5), \mathcal{L}^{\otimes k}(4\lambda))^{T},$$ defined by  $$ (z_1, z_2, \ldots, z_6) \mapsto (X_1^2, X_3^2, X_5^2, X_1X_3, X_1X_5, X_3X_5).$$ Let $I$ be the ideal of $A$ generated by the following relations: 
		\begin{center}
			\begin{tabular}{ |p{2.5cm}|p{2.5cm}|p{2.5cm}| }
				\hline
				$z_1z_2=z_4^2$   & $z_1z_3=z_5^2$    & $z_2z_3=z_6^2$ \\
				$z_1z_6=z_4z_5$&  $z_2z_5=z_4z_6$  & $z_3z_4=z_5z_6$\\ 
				\hline
			\end{tabular}.
		\end{center}
		
		Clearly, $I\subseteq \ker\phi$ and $\phi$ induces a homomorphism 
		$$\tilde{\phi}: A/I \longrightarrow\bigoplus\limits_{k \in \mathbb{Z}_{\geq 0}}H^{0}(X(v_5), \mathcal{L}^{\otimes k}(4\lambda))^{T}.$$
		
		Now, we show that $\tilde{\phi}$ is an isomorphism. 
		
		To complete the proof we use  the above relations as a reduction system. We show that diamond lemma (see \cite{bergman}) holds for this reduction system by looking at the reduction of the minimal overlapping ambiguities: $z_1z_2z_3, z_1z_2z_5, z_1z_2z_6, z_1z_3z_4, z_1z_3z_6, z_2z_3z_4, z_2z_3z_5.$
		
		Note that $z_1z_2z_3=z_3(z_1z_2)=z_3z_4^2 ~(\text{using } z_1z_2=z_4^2)=(z_3z_4)z_4=z_4z_5z_6 ~(\text{using } z_3z_4=z_5z_6).$
		
		Again $z_1z_2z_3=z_1(z_2z_3)=z_1z_6^2 ~(\text{using }z_2z_3=z_6^2)=(z_1z_6)z_6=z_4z_5z_6 ~(\text{using }z_1z_6=z_4z_5).$
		
		Also, $z_1z_2z_3=z_2(z_1z_3)=z_2z_5^2 ~(\text{using }z_1z_3=z_5^2)=(z_2z_5)z_5=z_4z_5z_6 ~(\text{using }z_2z_5=z_4z_6).$

		Likewise, we can show that the remaining overlapping ambiguities:
		
		$z_1z_2z_5, z_1z_2z_6, z_1z_3z_4, z_1z_3z_6, z_2z_3z_4, z_2z_3z_5$ reduce to unique expressions in $z_i$'s  namely, $z_4^2z_5,$ $z_4^2z_6,$ $z_4z_5^2,$ $z_5^2z_6,$ $z_4z_6^2,$ $z_5z_6^2$ after applying different reduction rules. Therefore, $\tilde{\phi}$ is an isomorphism. 
		
		Consider the embedding $$\psi: \mathbb{P}^2 \hookrightarrow \mathbb{P}^5$$ given by $$(x_1, x_2, x_3) \mapsto (x_1^2, x_2^2, x_3^2, x_1x_2, x_1x_3, x_2x_3).$$ Let $J$ be the homogeneous ideal of $A$ generated by $2 \times 2$ minors of the matrix $$\begin{pmatrix}
			z_1 & z_4 & z_5  \\
			z_4 & z_2 & z_6  \\
			z_5 & z_6 & z_3 \\
		\end{pmatrix},$$ where $z_i$'s are the homogeneous coordinates of $\mathbb{P}^5.$ Note that $\psi(\mathbb{P}^2)$ is given by $J.$ Further, note that $I \subseteq J$ and $dim (A/I)=dim(A/J)=3$ (computed using Macaulay2 \cite{GS}). Thus, the GIT quotient $T\backslash\backslash(X(v_5))^{ss}_{T}(\mathcal{L}(4\lambda))$ is isomorphic to $(\mathbb{P}^2, \mathcal{O}_{\mathbb{P}^2}(2))$ as a polarized variety. 
	\end{proof}

\subsection*{Acknowledgement.} Both the authors would like to thank the Indian Institute of Technology Bombay for the postdoctoral fellowships and the hospitality during their stay.

\end{document}